\documentclass[11pt, a4paper]{article}
\usepackage{authblk} 
\usepackage{amsthm}
\usepackage{amsmath}
\usepackage{amssymb}
\usepackage{enumerate}
\usepackage{verbatim}
\usepackage{nicefrac}
\usepackage{hyperref}
\usepackage{booktabs} 
\usepackage{cleveref}
\usepackage{xcolor}
\definecolor{darkblue}{rgb}{0,0,0.4}
\usepackage[skip=5pt plus1pt minus 1pt, indent=0pt]{parskip}
\usepackage[margin=1.3in]{geometry}
\usepackage[most, breakable]{tcolorbox}
\tcbset{sharp corners, boxrule=0.5pt,colback=white,colframe=darkblue!30, coltext=darkblue}

\usepackage[style=alphabetic, backend=biber]{biblatex}
\newtheorem{theorem}{Theorem}
\newtheorem{lemma}[theorem]{Lemma}
\newtheorem{proposition}[theorem]{Proposition}
\newtheorem{corollary}[theorem]{Corollary}

\newtheorem{fact}[theorem]{Fact}

\newtheorem{nlemma}{Lemma}[section]
\newtheorem{ntheorem}[nlemma]{Theorem}
\newtheorem{nproposition}[nlemma]{Proposition}
\newtheorem{nfact}[nlemma]{Fact}

\theoremstyle{remark}%
\newtheorem{remark}[theorem]{Remark}%
\newtheorem{nremark}[nlemma]{Remark}

\theoremstyle{definition}%

\newcommand{\Z}{\mathbb Z}

\newcommand{\N}{\mathbb N}
\newcommand{\F}{\mathbb F}

\newcommand{\Fq}{\F_q}
\newcommand{\Fqto}[1]{\F_{q^{#1}}}
\newcommand{\Fqk}{\F_{q^k}}
\newcommand{\FqX}{\Fq[X]}
\newcommand{\FqXrational}{\Fq(X)}

\newcommand{\cyclgen}[1]{\zeta_{#1}}
\newcommand{\of}[1]{\left(#1\right)}
\newcommand{\frobenius}[2]{{#1}^{(#2)}}

\DeclareMathOperator{\ord}{ord}
\DeclareMathOperator{\rad}{rad}
\DeclareMathOperator{\lcm}{lcm}
\newcommand{\coeffdeg}[2]{\mathrm{coeffdeg}_{#1}\left(#2\right)}
\renewcommand{\char}{\text{char}}
\newcommand{\cyclocoset}[3]{\mathrm C_{#1,#2}(#3)}
\newcommand{\cyclorepsystem}[2]{\mathrm {CR}_{#1}(#2)}
\newcommand{\spin}[2]{\mathrm{spin}_{#1}\left[#2\right]}

\newcommand{\blue}[1]{{\color{darkblue}#1}}

\newcommand{\changed}[1]{\blue{#1}}

\providecommand{\keywords}[1]{\small \noindent \textbf{Keywords:} #1}

\begin{document}
	
	\title{Closed formulas for the factorization of $X^n-1$, \\
		the $n$-th cyclotomic polynomial, $X^n-a$ and $f(X^n)$\\
		 over  a finite field 
		 for arbitrary positive integers $n$}
	
	\author{Anna-Maurin Graner\\
		{University of Rostock}, Germany\\
		\tt {amg.research@posteo.com}}

	\maketitle
	
	\begin{abstract}
		The factorizations of the polynomial \(X^n-1\) and the cyclotomic polynomial \(\Phi_n\) over a finite field \(\Fq\) have been studied for a very long time. Explicit factorizations have been given for the case that \(\rad(n)\mid q^w-1\) where \(w=1\), \(w\) is prime or \(w\) is the product of two primes. For arbitrary \(a\in \Fq^\ast\) the factorization of the  polynomial \(X^n-a\)  is needed for the construction of constacyclic codes. Its factorization has been determined for the case \(\rad(n)\mid q-1\) and for the case that there exist at most three distinct prime factors of \(n\) and \(\rad(n)\mid q^w-1\) for a  prime \(w\). Both polynomials \(X^n-1\) and \(X^n-a\) are compositions of the form \(f(X^n)\) for a monic irreducible polynomial \(f\in \FqX\). The factorization of the composition \(f(X^n)\) is known for the case \(\gcd(n, \mathrm{ord}(f)\cdot \deg(f))=1\) and \(\rad(n)\mid q^w-1\) for \(w=1\) or \(w\) prime.   
		
		However, there does not exist a closed formula for the explicit factorization of either  \(X^n-1\), the cyclotomic polynomial \(\Phi_n\), the binomial \(X^n-a\) or the composition \(f(X^n)\). Without loss of generality we can assume that \(\gcd(n,q)=1\). Our main theorem, \Cref{theorem: factorization X^n-a for gcd(n_q)=1}, is a closed formula for the factorization of \(X^n-a\) over \(\Fq\) for any \(a\in \Fq^\ast\) and any positive integer \(n\) such that \(\gcd(n,q)=1\). From our main theorem we derive one closed formula each for the factorization of \(X^n-1\)  and of the \(n\)-th cyclotomic polynomial \(\Phi_n\)  for any positive integer \(n\) such that \(\gcd(n,q)=1\) (\Cref{theorem: factorization of X^n-1 for gcd(n q)=1} and \Cref{theorem: factorization cyclotomic polynomial}).  Furthermore,  our main theorem yields a closed formula for the factorization of the composition \(f(X^n)\) for any irreducible polynomial \(f\in \FqX\), \(f\neq X\), and any positive integer \(n\) such that \(\gcd(n,q)=1\) (\Cref{theorem: factorization of f(X^n) for gcd(n q)=1}).
	\end{abstract}


\section{Introduction}
\label{section: Intro}

\changed{This preprint has been updated with significant changes. To make these visible, we present them in blue color. Most of them can be found in Sections \ref{subsection: factorization X^n-a gcd(n q)=1}, \ref{subsection: factorization X^n-1 and cyclotomic polynomial} and \ref{section: factorization of f(X^n)}. New results are numbered as \(S.N\), where \(S\) is the section number and \(N\in \N\), so that the old theorem numbering stays intact. Theorems 19, 24, and 25 have been removed, because  no value is added by including complicated factorizations for special cases when a general formula is given.}

Let \(\Fq\) be a finite field of \(q\) elements and of characteristic \(\char(\Fq)\). The algebraic closure of \(\Fq\) is \(\overline{\F}_q\). For every \(a\in \Fq^\ast\) we denote by \(\ord(a)\) the order of \(a\) in the multiplicative group \(\Fq^\ast\). Furthermore, for an irreducible polynomial \(f\in \FqX\)  the smallest positive integer \(e\) such that \(f \mid X^e-1\)  or, equivalently, the multiplicative order of all of its roots, is called the \textit{order} of \(f\) and is denoted by \(e = \ord(f)\). Let \(n\) be a positive integer  and \(n=p_1^{i_1}\cdots p_m^{i_m}\) its decomposition into powers of distinct primes \(p_1, \ldots, p_m\). Then \(\nu_{p_j}(n) = i_j\) denotes  the \(p_j\)-adic value of \(n\) and we call \(\rad(n):= p_1 \ldots p_m\) the \textit{radical of \(n\)}. If \(\gcd(n,q)=1\), then there exists a primitive \(n\)-th root of unity in the splitting field \(\overline \F_q\), which we denote by \(\cyclgen{n}\). For a positive integer \(m\) such that \(\gcd(n,m)=1\) the multiplicative order of \(m\) in \(\Z/n\Z\) is \(\ord_n(m)\). Thus, \(\cyclgen{n}\) is an element of the extension field \(\Fqto{\ord_n(q)}\). Let \(g,h\in \FqX\) be polynomials over \(\Fq\) such that \(h\neq 0\), then \(Q = \frac gh\) is called a \textit{rational function} and the set of all rational functions over \(\Fq\) is denoted by \(\FqXrational\). For a polynomial \(f\in \FqX\) and \(Q\in \FqXrational\) we define the \textit{\(Q\)-transform} or \textit{rational transformation of \(f\) with \(Q\)} as \(f^Q:= h^n \cdot f\of{\frac gh}\in \FqX\).

The factorization of polynomials over a finite field \(\Fq\) is needed for many applications in cryptography and coding theory. In particular, the factorization   of the polynomial \(X^n-a\) for \(a\in \F_q^\ast\) is used for the construction of \(a\)-constacyclic codes of length \(n\). An \textit{\(a\)-constacyclic code of length \(n\)} is an ideal in \(\FqX/\langle X^n-a\rangle\), which is a principal ideal generated by a factor of \(X^n-a\). For many special cases the factorization of \(X^n-a\) has been determined (see for example \cite{Dinh2012, BakshiRaka2012, ChenFanLinLiu2012, LiuLiWang2017, LiYue2018, WuYue2018constacyclic, ShiFu2020, RakphonChongchitmatePhuto2022}). In particular, \cite{WuYue2018constacyclic} gave the factorization of \(X^n-a\) for the case \(\rad(n)\mid q-1\) and \cite{RakphonChongchitmatePhuto2022} for the case that \(n\) has up to three distinct prime factors and \(\rad(n)\mid q^w-1\) for a prime \(w\).  If \(a =1\), then the  ideals in \(\FqX/\langle X^n-1\rangle\) are the \textit{cyclic codes of length \(n\)}. Every cyclic code is generated by a factor of the polynomial \(X^n-1\). Many results on the factorization of \(X^n-1\) exist (see \cite{ BlakeGaoMullin1993,LN1994, Meyn1996, Brochero-MartinezGiraldo-VergaradeOliveira2015, WuYueFan2018, OliveiraReis2021, SinghDeepak2023}).  In particular, \cite{Brochero-MartinezGiraldo-VergaradeOliveira2015} determined the factorization of \(X^n-1\) for the case \(\rad(n)\mid q-1\) and  \cite{WuYueFan2018} and \cite{WuYue2021} gave the factorization of \(X^n-1\) for the case that \(\rad(n)\) divides \(q^w-1\) or \(q^{vw}-1\), where \(v\) and \(w\) are prime. 

The factors of \(X^n-1\) called the cyclotomic polynomials have received a lot of attention. For every positive integer \(n\) such that \(\gcd(n,q)=1\) the \textit{\(n\)-th cyclotomic polynomial \(\Phi_n\)} is defined as \(\Phi_n = \prod_{\substack {0 \leq j \leq n-1\\ \gcd(j,n)=1}} (X-\cyclgen n^j).\) It is the product of all \(\varphi(n)\) primitive \(n\)-th roots of unity. From the fact that  \(X^n-1\) decomposes as \(\prod_{j=0}^{n-1} (X-\cyclgen n^j)\) over \(\overline{\F}_q\) follows directly  that \(X^n-1 = \prod_{d\mid n} \Phi_d\). Thus, if the factorization of \(\Phi_d\) is known for every divisor \(d\) of \(n\), then these factorizations yield the factorization of \(X^n-1\). It is well known that  \(\Phi_n\)  factors into \(\frac {\varphi(n)} {\ord_n(q)}\) distinct monic irreducible polynomials of degree \(\ord_n(q)\) over \(\Fq\). More explicit factorizations of \(\Phi_n\) have been given in \cite{Stein2001,FY2007,WangWang2012, TW2013,WZFY2017, Alshareef2018}.

 The second polynomial which we study in this paper is the composition \(f(X^n)\) for an irreducible polynomial \(f\in \FqX\) and a positive integer \(n\). This composition is in fact a rational transformation  because \(f(X^n)=f^Q\) for \(Q=\frac {X^n}1\in \FqXrational\). Rational transformations are widely used for the construction of irreducible polynomials \cite{Wiedemann1988, Meyn1990, McNay1995, Kyuregyan2002, PanarioReisWang2020}. The classical composition method searches for a rational function \(Q=\frac gh\in \FqXrational\) such that the rational transformation \(f^Q\) is irreducible. Variations of this method search for irreducible factors of the composition \(f^Q\) (see \cite{Ugolini2013,Ugolini2016}).

Let  \(f\in \FqX\) be an irreducible polynomial of degree \(k\) and let  \(\alpha\in \Fqk\) be a root of \(f\).   Then the explicit factorization of \(f(X^n)\) is known for the case \(\gcd(n, \ord(\alpha)\cdot k)=1\) and \(\rad(n)\mid q^w-1\) for \(w=1\) or \(w\) prime (see \cite{Brochero-MartinezReisSilva-Jesus2019}). The authors of \cite{Brochero-MartinezReisSilva-Jesus2019} determined the factorization of \(f(X^n)\) directly. This is not necessary, because the  factorization of \(f(X^n)\) can be obtained from the factorization of the polynomial \(X^n-\alpha\) over \(\Fqto{k}\), as the following three theorems show.  For every rational function \(Q=\frac gh \in \FqXrational\) the \(Q\)-transform \(f^Q\)  is directly connected to the polynomial \(g-\alpha h\) over \(\Fqto{k}\), as we can see from the following well-known theorem by Cohen.

\begin{theorem}[{\cite[Lemma 1]{Cohen1969}}]
	\label{Cohen1969: Lemma 1}
	Let  \(f\in \FqX\) be an irreducible polynomial of degree \(k\) and \(\alpha\in \Fqto k\) be a root of \(f\). Further, let \(Q = \frac gh \in \FqXrational\) be a rational function over \(\Fq\). Then the polynomial  \(f^Q\) is irreducible over \(\Fq\) if and only if  the polynomial \(g-\alpha h\) is irreducible over \(\Fqto{k}\).
\end{theorem}

Kyuregyan and Kyureghyan gave a proof   for \Cref{Cohen1969: Lemma 1} in \cite{Kyuregyan2011}  which can be tweaked  to show that the factorization of \(g-\alpha h\) yields the factorization of \(f^Q\). This fact is stated in \Cref{Mullin2010: Lemma 13}, which was presented in \cite{Mullin2010}. Since the paper is not publicly available, we present the proof of \Cref{Mullin2010: Lemma 13} here.   For its formulation and its proof we need the following definitions and \Cref{KK11: Lemma 1}. 

For positive integers \(j\) and \(d\), and a polynomial \(h (X)=\sum_{i=0}^m a_i X^i\) over \(\Fqto d\), we define the polynomial 
\[\frobenius h j (X) := \sum_{i=0}^m a_i^{q^j} X^i.\]
Note that \(h^{(j)}=\frobenius h 0 = h\) if \(j\) is a multiple of \([\Fq(a_0,\ldots, a_m):\Fq]\), the extension degree of \(\Fq(a_0, \ldots, a_m)\) over \(\Fq\). The extension degree \([\Fq(a_0,\ldots, a_m):\Fq]\) plays a key role in the arguments to come. {We denote it by \(\coeffdeg qh\) and call it the \textit{degree of the coefficients of \(h\) over \(\Fq\)}. Note that in other publications the polynomial \(\frobenius hj \) is also denoted by \(\sigma_q^j(h)\), where \(\sigma_q\) is the Frobenius automorphism (as a ring autormorphism on \(\FqX\)).

\begin{theorem}[{\cite[Lemma 1]{Kyuregyan2011}}]
	\label{KK11: Lemma 1}
	Let \(g\in \FqX\) be a monic polynomial of degree \(dm \), where \(d,m\in \N\).  Then \(g\) is irreducible over \(\Fq\) if and only if there exists a monic irreducible polynomial \(h\) of degree \(m\) over \( \Fqto d\)  such that \(\coeffdeg qh=d\) and 	\[g = \prod_{j=0}^{d-1} \frobenius h j.\]
\end{theorem}

From \Cref{KK11: Lemma 1} follows directly, that for a monic irreducible polynomial \(h\) of degree \(m\) over \(\Fqto{d}\) such that \(\coeffdeg qh=d\) the polynomial \(\prod_{j=0}^{d-1} h^{(j)}\) is a monic irreducible polynomial of degree \(dm\) over \(\Fq\). We call this polynomial the \textit{\(q\)-spin of \(h\)} and denote it by \(\spin  q h\). Note that if \(\beta \in \Fqto{dm}\) is a root of \(h\), then \(h\) is the minimal polynomial of \(\beta\) over \(\Fqto d\) and \(g\) is the minimal polynomial of \(\beta\) over \(\Fq\). With the definition of the \(q\)-spin we can state \Cref{Mullin2010: Lemma 13}.

\begin{theorem}[{\cite[Lemma 13]{Mullin2010}}]
	\label{Mullin2010: Lemma 13}
	Let  \(f\in \FqX\) be an irreducible polynomial of degree \(k\) and  \(\alpha\in \Fqto k\) be a root of \(f\). Further, let \(Q= \frac g h\in \FqXrational\) and \(\prod_R R\) be the factorization of \(g-\alpha h\) into irreducible factors over \(\Fqto k\). Then the factorization of \(f^Q\) into irreducible factors  over \(\Fq\) is given by 
	\begin{equation*}
		\prod_{R} \spin  q R,
	\end{equation*}
	where \(\coeffdeg q R=k\) for all irreducible factors \(R\) of \(g-\alpha h\) over \(\Fqto k\).
\end{theorem}

\begin{proof}[Proof of \Cref{Mullin2010: Lemma 13}]
	Since \(f\) is an irreducible polynomial of degree \(k\) and \(\alpha\) is a root of \(f\), the decomposition of \(f\) over \(\Fqto k\) is \(f(X) = \prod_{j=0}^{k-1} (X-\alpha^{q^j})\). Thus, the \(Q\)-transform of \(f\) considered over \(\Fqto{k}\) satisfies:
	\begin{equation*}
		\label{eq: f^Q and g-alpha h}
		f^Q = h^k \cdot f\of{\frac gh} = \prod_{j=0}^{k-1} (g-\alpha^{q^j}\cdot h) = \prod_{j=0}^{k-1} \frobenius{(g-\alpha h)} j = \prod_{j=0}^{k-1} \frobenius {\left(\prod_R R\right) } j = \prod_{j=0}^{k-1} \prod_R \frobenius R j,
	\end{equation*}
	because \(g, h\in \FqX\) and therefore \(\frobenius gj =g\) and \(\frobenius h j=h\) for all \(0 \leq j \leq {k-1}\). 
	Furthermore, if \(\prod_R R\) is the factorization of \(g-\alpha h\) over \(\Fqto{k}\), then for every irreducible factor \(R\) the degree of its coefficients over \(\Fq\) must be equal to \(k\). Indeed, if \(\coeffdeg{q}{R}=m<k\) and \(\gamma \in \Fqto{m}\) were a root of \(R\), then \(\gamma\) would be a root of \(g-\alpha h\) satisfying \(g(\gamma)-\alpha \cdot h(\gamma)=0\).  However \(g(\gamma)\) and \(h(\gamma)\) are elements of \(\Fqto m\), which implies that \(\alpha \cdot h(\gamma)\) cannot be equal to \(g(\gamma)\), a contradiction. Consequently, the \(q\)-spin of \(R\) is equal to \(\spin q R = \prod_{j=0}^{k-1} \frobenius R j \) and 
	\[f^Q = \prod_R \left(\prod_{j=0}^{k-1} \frobenius R j\right) = \prod_R \spin q R.\]
\end{proof}
}

From \Cref{Mullin2010: Lemma 13} follows directly that it suffices to study the factorization of  \(g-\alpha h\) over \(\Fqto{k}\) to obtain the factorization of \(f^Q\). For the rational function \(Q=\frac {X^n} 1\), this is the polynomial \(X^n-\alpha\). The case  \(\alpha=1\) only applies if \(f=X-1\). However, if there exists an element \(\beta\) in \(\Fqto{k}\) with \(\beta ^n = \alpha\), then the factorization of \(X^n-\alpha\) can easily be obtained from the factorization of \(X^n-1\). In terms of constacyclic codes, the \(\alpha\)-constacyclic code is \(n\)-equivalent to a cyclic code (see \cite[Lemma 3.1]{Hughes2000}, \cite{ChenDinhLiu2014}). We reformulate this well-known connection using the concepts introduced above:

\begin{lemma}
	\label{lemma: X^n-alpha from X^n-1 for beta^n = alpha}
	Let  \(n\in \N\) and \(f\in \FqX\) be an irreducible polynomial of degree \(k\). Further, let \(\alpha\in \Fqto k\) be a root of \(f\) such that there exists \(\beta\in \Fqto{k}\) with \(\beta^n = \alpha\). If  \(\prod_{R} R\) is the factorization of \(X^n-1\) into monic irreducible factors over \(\Fqto k\), then the factorization of \(X^n-\alpha\) into monic  irreducible factors over \(\Fqto{k}\) is \(\prod_R R^Q\), where \(Q = \frac X\beta \in \Fqto k (X)\). 
\end{lemma}

\begin{proof}
	Since \(\beta^n = \alpha\), we can write 
	\begin{equation*} 
		X^n-\alpha = X^n - \beta^n = \beta^n \left(\left(\frac X \beta \right)^n - 1\right) = \beta ^n \prod_{R} R\of {\frac X \beta} = \prod_{R} \beta^{\deg(R)} R\of{\frac X \beta } = \prod_{R} R^Q.
	\end{equation*}
	For every monic irreducible factor \(R\) let \(\gamma_R\) be a root of \(R\). Then the polynomial \(X-\gamma_R \cdot \beta\) is irreducible over \(\Fqto{k\cdot \deg(R)}\). Thus, with   \Cref{Cohen1969: Lemma 1} the polynomial \(R^Q\) is irreducible over \(\Fqto k\), which concludes our proof.
\end{proof}

We combine \Cref{lemma: X^n-alpha from X^n-1 for beta^n = alpha} with \Cref{Mullin2010: Lemma 13} and show that if there exists an element \(\beta\in \Fqto k\) such that \(\beta^n = \alpha\), then the factorization of \(f(X^n)\) is given by the factorization of \(X^n-1\):

\begin{proposition}
	\label{theorem: factorization of f(X^n) if beta^n=alpha}
	Let  \(n\in \N\) and \(f\in \FqX\) be an irreducible polynomial of degree \(k\). Further, let \(\alpha\in \Fqto k\) be a root of \(f\) such that there exists \(\beta\in \Fqto{k}\) with \(\beta^n = \alpha\). If \(\prod_{R} R\) is  the factorization of \(X^n-1\) into monic irreducible factors over \(\Fqto k\). Then the factorization of \(f(X^n)\) into monic  irreducible factors over \(\Fq\) is
	\[\prod_{R} \spin q {R^Q} ,\]
	where \(Q=\frac X \beta\in \Fqto k (X)\).
\end{proposition}
 
 It is well known that  there exists an element \(\beta\in \Fqto k\) such that \(\beta^n=\alpha\) if and only if \(\alpha^{\frac{q^k-1}{\gcd(n,q^k-1)}} = 1\). If the order of \(\alpha\) is large and \(n\) is not coprime with \(q^k-1\), it is likely that such an element \(\beta\) does not exist. Therefore, we need to find the factorization of \(X^n-\alpha\) over \(\Fqto k\) for \(\alpha \neq 1\) in general. 
 
 \changed{Many authors include the number of irreducible factors in the statement of their factorizations. However, the number of irreducible factors of \(f(X^n)\) over \(\Fq\)  has been determined in \cite{Butler1955} and we omit this number in our results. 
 
 \begin{ntheorem}[\cite{Butler1955}]\label{Butler1955}
 	Let \(n\) be a positive integer such that \(\gcd(n,q)=1\) and  \(f\in \FqX\) be an irreducible polynomial  over \(\Fq\) of degree \(k\) and order \(e\). Further, let \(n = n_1\cdot n_2\), where \(\rad(e)\mid n_1\) and \(\gcd(n_2,e)=1\). Then
 	\begin{enumerate}[(i)]
 		\item Each root of \(f(X^n)\) has order \(d\cdot n_1\cdot e\) for a divisor \(d\) of \(n_2\).
 		\item For every divisor \(d\) of \(n_2\), there exist exactly \(k\cdot n_1 \cdot \frac{\varphi(d)}{\ord_{d \cdot n_1 \cdot e}(q)}\) irreducible factors of \(f(X^n)\) over \(\Fq\). Each of these factors has degree \(\ord_{d \cdot n_1 \cdot e}(q)\).
 	\end{enumerate}
 \end{ntheorem} }

 \paragraph{Structure of this paper:}
 Our main result, \Cref{theorem: factorization X^n-a for gcd(n_q)=1}, is a closed formula for the factorization of \(X^n-a\) into monic irreducible polynomials over \(\Fq\) for any positive integer \(n\) and any element \(a\) of \(\Fq^\ast\).  First, we give a new proof for the factorization of \(X^n-a\) for the case \(\rad(n)\mid q-1\) and (\(4\nmid n\) or \(q\equiv 1 \pmod 4\)) in \Cref{subsection: factorization X^n-a special case rad(n) mid q-1}. This proof  yields one closed formula for this case, \Cref{theorem: factorization of X^n-a for rad(n) mid q-1_4 nmid n or q = 1 mod 4}. In \cite{WuYue2018constacyclic} the factorization for this case is given in three separate theorems and our \Cref{theorem: factorization of X^n-a for rad(n) mid q-1_4 nmid n or q = 1 mod 4}  combines and extends all of them \cite[Theorems 6(1), 8(1), 9(1)]{WuYue2018constacyclic}.  In \Cref{subsection: factorization X^n-a gcd(n q)=1} we combine \Cref{theorem: factorization of X^n-a for rad(n) mid q-1_4 nmid n or q = 1 mod 4} and \Cref{KK11: Lemma 1} to obtain a closed formula for the factorization of \(X^n-a\) for any positive integer \(n\) such that \(\gcd(n,q)=1\) and any \(a\in \Fq^\ast\). This is our main theorem,  \Cref{theorem: factorization X^n-a for gcd(n_q)=1}. In \Cref{subsection: factorization X^n-1 and cyclotomic polynomial} we derive a closed formula for the factorization of \(X^n-1\) (see \Cref{theorem: factorization of X^n-1 for gcd(n q)=1}) and a closed formula for the factorization of \(\Phi_n \) for any positive integer \(n\) such that \(\gcd(n,q)=1\)  from our main theorem. 
 Furthermore, in \Cref{section: factorization of f(X^n)} we combine our main theorem and \Cref{Mullin2010: Lemma 13} and give a closed formula for the  factorization of \(f(X^n)\) for any irreducible polynomial \(f\in \FqX\), \(f\neq X\), and any positive integer \(n\) such that \(\gcd(n,q)=1\) (see \Cref{theorem: factorization of f(X^n) for gcd(n q)=1}).

The results on the factorization of \(X^n-1\), \(X^n-a\) and \(f(X^n)\) given in \cite{Brochero-MartinezGiraldo-VergaradeOliveira2015, WuYueFan2018,WuYue2018constacyclic, WuYue2021, Brochero-MartinezReisSilva-Jesus2019} are proved with a top-down approach. More precisely, for a polynomial \(g\) over \(\Fq\), the authors first name the irreducible factors over \(\Fq\). Then, they show that the named polynomials are irreducible over \(\Fq\) and factors of \(g\). In the last step, the argument that the sum of the degrees of the named polynomials equals the degree of \(g\) concludes their proof.

We use a bottom-up approach and prove our results in \Cref{section: factorization of X^n-a} via induction.  More precisely,  if \(n=p_1\cdots p_m\) is the factorization of \(n\) into prime factors, we decompose \begin{equation*}
	X^{p_1}-\alpha, \; \left(X^{p_2}\right)^{p_1} -\alpha, \;\left( \left(X^{p_3}\right)^{p_2}\right)^{p_1}-\alpha, \ldots, \left(\ldots  \left(X^{p_m} \right)^{p_{m-1}} \ldots \right) ^{p_1}-\alpha 
\end{equation*} inductively. The advantage of this method is that the proof itself shows why the factors are given as they are \changed{and it allows us to give one closed formula for all \(n\in \N\) such that \(\gcd(n,q)=1\) and all \(a\in \Fq^\ast\).}

\section{The factorization of $X^n-a$}
\label{section: factorization of X^n-a}

Let \(a\) be an element of \(\Fq^{\ast}\). \changed{Then the factorization of \(X^n-a\) is known for positive integers \(n\) such that \(\rad(n)\mid q-1\) (see \cite[Theorem 6, Theorem 8, Theorem 9]{WuYue2018constacyclic})} or \(\rad(n)\mid q^w-1\) for a prime \(w\) and \(n=p_1^{m_1}\cdot p_2^{m_2}\cdot p_3^{m_3}\), where \(p_1,p_2,p_3\) are distinct primes (see \cite{RakphonChongchitmatePhuto2022}). In this section we present the explicit factorization of \(X^n-a\) over \(\Fq\) for any positive integer \(n\). 

\begin{nremark}\label{remark: gcd(n q)>1 for X^n-a}
	Without loss of generality we can assume that \(\gcd(n,q)=1\). Indeed, if \(\gcd(n,q)>1\), then \(n=\tilde n \cdot \char(\Fq)^{l}\) for positive integers \(\tilde n , l\). Then there exists an element \(b \in \Fq\) such that \(b^{\char(\Fq)^l}=a\) and   the polynomial \(X^n-a\) equals \((X^{\tilde n}-b)^{\char(\Fq)^l}\).
\end{nremark}

We prove our results by induction on the number of  (not necessarily distinct) prime factors of \(n\). As base case we study the polynomial \(X^p-a\) for a prime \(p\) such that \(\gcd(p,q)=1\). Then there exists an element \(b\in \overline \F_q\) such that   \(b^p=a\) and  the  polynomial \(X^p-a\) splits as  \(\prod_{j=0}^{p-1} (X-\cyclgen p^j b)\), where \(\cyclgen{p}\) is a \(p\)-th primitive root of unity over \(\Fq\). By \Cref{Mullin2010: Lemma 13}, the factorization of \(X^p-a\) over \(\Fq\) depends on the  degree of the coefficients of \((X-\cyclgen p^j b)\) over \(\Fq\). Since \(\coeffdeg{q}{X-\cyclgen p^j b} = \ord_{\ord(\cyclgen p^j b)}(q)\), it is determined by the order of the roots \(\cyclgen p^j b\). The following proposition specifies this order for the two cases \(p\mid \ord(a)\) and \(p \nmid \ord(a)\).

\begin{proposition}\label{proposition: ord(b) s.t. b^p=a}
	Let \(a\in \Fq^{\ast}\) and \(p\) prime such that \(\gcd(p,q)=1\). Then the following statements hold: 
	\begin{enumerate}[(i)]
		\item If \(p \nmid \ord(a)\), then there exists a positive integer \(r\) such that   \((a^r)^p =a\) and \(\ord(a^r)=\ord(a)\). \changed{If \(a=1\), then \(r=1\), else \(r\) satisfies \(rp\equiv 1\pmod{\ord(a)}\).} Furthermore, \(\ord(\cyclgen p^j a^r)=p  \cdot \ord(a)\) for all \(1\leq j \leq p-1\). \label{proposition: b^p=a case p nmid ord(a)}
		\item If \(p \mid \ord(a)\), then every root \(b\in \overline \F_q\) of the polynomial \(X^p-a\) has order \(\ord(b)=p\cdot \ord(a)\). 
		\label{proposition: b^p=a case p mid ord(a)}
	\end{enumerate}
\end{proposition}

\begin{proof}
	Note that for any element \(b \in \overline{ \F}_q\) such that \(b^p=a\), the order of \(b\) satisfies 
	\begin{equation}\label{eq: order of b s.t. b^p=a}
		\ord(b)= \ord(a)\cdot \gcd(\ord(b),p) \in \{\ord(a), \ord(a)\cdot p\}.
	\end{equation}
	\begin{enumerate}[(i)]
		\item If \(p \nmid \ord(a)\) \changed{and \(\ord(a)>1\),} then there exists \(r\in \N\) such that \(rp\equiv 1 \pmod {\ord(a)}\) and the element \(a^r\in \Fq\) satisfies \((a^r)^p = a\). \changed{If \(a=1\), then obviously \((a^1)^p=a\) and we set \(r=1\).} Since \(a^r\) is an element of the multiplicative group \(\langle a \rangle\leq \Fq^{\ast}\), its order divides \(\ord(a)\). This fact combined with \cref{eq: order of b s.t. b^p=a} yields that \(\ord(a^r)=\ord(a)\). For every \(1 \leq j \leq p-1\), the element \(\cyclgen{p}^j\)  has order \(p\), which is coprime with \(\ord(a)\). Therefore, we can conclude that \(\ord(\cyclgen p^j a^r)=p \cdot \ord(a)\) for \(1\leq j \leq p-1\).
		\item If \(p \mid \ord(a)\), then with \cref{eq: order of b s.t. b^p=a} also \(p \mid \ord(b)\). Suppose that \(\ord(b)=\ord(a)\). With the fact that \(a = b^p\), we know that \(\ord(a) = \frac {\ord(b)}{\gcd(\ord(b),p)} = \frac {\ord(b)} p = \frac {\ord(a)} p\), a contradiction. Consequently, our assumption must have been false and we have \(\ord(b)= \ord(a)\cdot p\).
	\end{enumerate}
\end{proof}

Note that, with \Cref{proposition: ord(b) s.t. b^p=a}, for  a prime \(p\) such that \(p \mid q-1\) and \(p \nmid \ord(a)\), the polynomial \(X^p-a\) always decomposes in \(\FqX\). Furthermore, in both cases, \(p \nmid \ord(a)\) and \(p \mid \ord(a)\), there exists a root of \(X^p-a\) of order \(p \cdot \ord(a)\). We can easily derive the following corollary.

\begin{corollary}\label{theorem: existence of b s.t. b^p=a in Fq}
	Let \(a\in \Fq^{\ast}\) and  \(p\) prime such that \(p \mid (q-1)\). Then there exists \(b\in \Fq\) such that \(b^p=a\) if and only if \(p \cdot \ord(a)\mid (q-1)\).
\end{corollary}

\begin{proof}
	If \(p \cdot \ord(a)\mid (q-1)\), then with \Cref{proposition: ord(b) s.t. b^p=a} all roots of \(X^p-a\) are elements of \(\Fq\). It remains to argue that if one root \(b\) of \(X^p-a\) lies in \(\Fq\), then all roots \(\cyclgen p^j b\) for \(0\leq j \leq p-1\) are elements of \(\Fq\), because \(\cyclgen p \in  \Fq\). Thus, the order of every root of \(X^p-a\) divides \(q-1\) and in particular the order \(p \cdot \ord(a)\).
\end{proof}

As the following theorem by Serret, \Cref{Serret1866: Theorem Irreducibility of X^n-a}, shows, the irreducibility of the binomial \(X^t-a\) for a positive integer \(t\) depends mainly on the order of \(a\). Since \Cref{proposition: ord(b) s.t. b^p=a} and \Cref{theorem: existence of b s.t. b^p=a in Fq} give the order of the roots of \(X^p-a\), they are the key elements for the theorems to come.  
	
	\begin{theorem}[{\cite[Theorem 3.75]{LN1994}}]
		\label{Serret1866: Theorem Irreducibility of X^n-a}
		Let \(t\geq 2\). Then the binomial \(X^t-a\) is irreducible in \(\Fq[X]\) if and only if  the following three conditions are satisfied:
		\begin{enumerate}[(i)]
			\item \(\rad(t)\mid \ord(a)\),
			\item \(\gcd( t, \frac {q-1} {\ord(a)}) = 1\),
			\item \(4\nmid t\) or  \(q\equiv 1 \pmod 4\).
		\end{enumerate}
	\end{theorem}

\subsection{The factorization of $X^n-a$ for positive integers $n$\\ such that  $\mathrm{rad}(n)\mid q-1$ and  ($4\nmid n$ or $q\equiv 1 \mod 4$)}
\label{subsection: factorization X^n-a special case rad(n) mid q-1}

\changed{The factorization of \(X^n-a\)  for the case \(\rad(n)\mid q-1\) and (\(4\nmid n\) or \(q\equiv 1 \pmod 4\)) from  \cite{WuYue2018constacyclic} is given in three separate theorems for three types of \(a\in \Fq^\ast\).  In this section we give a new proof for this case which, due to its inductive nature, yields one closed formula for all \(a\in \Fq^\ast\).} 

In the next theorem we combine \Cref{theorem: existence of b s.t. b^p=a in Fq} and \Cref{Serret1866: Theorem Irreducibility of X^n-a} to obtain a result which can be applied recursively. More precisely, given an irreducible polynomial \(X^t-a\) over \(\Fq\), \Cref{theorem: irreducibility of X^(tp)-a}  gives an easy irreducibility criterion for the polynomial \(X^{tp}-a\), where \(p\) is a prime and divides \(q-1\). 

\begin{proposition}\label{theorem: irreducibility of X^(tp)-a}
	Let \(a\in \Fq^\ast\), \(p\) be a prime such that \(p \mid (q-1)\) and let \(t\in \N\) such that \(X^t-a\) is irreducible over \(\Fq\). Further, let  \(4\nmid tp\) or  \(q \equiv 1 \pmod 4\). Then \(X^{tp}-a\) is irreducible over \(\Fq\) if and only if \(p \cdot \ord(a)\nmid (q-1)\).
\end{proposition}

\begin{proof}
	We show that \(X^{tp}-a\) is irreducible over \(\Fq\) if and only if there does not exist \(b\in \Fq\) such that \(b^p=a\). Then \Cref{theorem: existence of b s.t. b^p=a in Fq} completes the statement. 
	
	Suppose that there exists \(b\in \Fq\) such that \(b^p=a\). Then the polynomial \(X^{tp}-a\) in \(\FqX\) satisfies \(X^{tp}-a = \prod_{j=0}^{p-1} (X^t-\cyclgen p^j b)\) and it is reducible.
	
	Suppose that there does not exist \(b\in \Fq\) such that \(b^p=a\). Then all roots of \(X^p-a\)  lie in true extension fields of \(\Fq\). Let  \(b\in \Fq(b)\) be a root of \(X^p-a\). Then since \(\cyclgen p\in \Fq\), we have \(\cyclgen{p}^j b \in \Fq(b)\) for all \(1\leq j \leq p-1\). Vice versa, \(b= \cyclgen p^{-j} \cdot \cyclgen p^jb \in \Fq(\cyclgen p^j b)\). Thus, the degree of the minimal polynomial of every root of \(X^p-a\) is equal to \([\Fq(b):\Fq]\).  Since \(X^p-a\) is a product of the minimal polynomials of its roots over \(\Fq\), there exists a positive integer \(m\) such that \(p= m\cdot [\Fq(b):\Fq]\). The two facts that \(p\) is prime and \(b\) is not an element of \(\Fq\) imply that \([\Fq(b):\Fq]=p\) and \(X^p-a\) is irreducible over \(\Fq\).
	
	\Cref{Serret1866: Theorem Irreducibility of X^n-a}  yields that \(p \mid \ord(a)\) and \(\gcd(p, \frac {q-1}{\ord(a)})=1\). Furthermore, since \(X^t-a\) is irreducible over \(\Fq\), we have \(\rad(t)\mid \ord(a)\), \(\gcd(t, \frac {q-1}{\ord(a)})=1\). As a direct consequence we obtain that  \(\rad(tp)\mid \ord(a)\) and also \(\gcd(tp, \frac {q-1}{\ord(a)})=1\). Since we assumed that \(4 \nmid tp\) or \(q\equiv 1 \pmod 4\), the third condition of \Cref{Serret1866: Theorem Irreducibility of X^n-a} is satisfied and we can conclude that \(X^{tp}-a\) is irreducible over  \(\Fq\).
\end{proof}

Let \(p\) be a prime such that \(4\nmid tp\) or \(q\equiv 1 \pmod 4\). Then the following corollary states the useful fact that if \(X^t-a\) is irreducible and \(p \mid t\), the polynomial \(X^{tp}-a\) is irreducible.

\begin{corollary}\label{theorem: irreducibility of X^(tp)-a with p mid t}
	Let \(a\in \Fq^\ast\), \(p\) be a prime such that \(p \mid (q-1)\) and let \(t\in \N\) such that \(X^t-a\) is irreducible over \(\Fq\). Further, let  \(4\nmid tp\) or  \(q \equiv 1 \pmod 4\).  Then if \(p\) divides \(t\) the polynomial \(X^{tp}-a\) is irreducible over \(\Fq\).
\end{corollary}

\begin{proof}
	Since \(X^t-a\) is irreducible over \(\Fq\), \Cref{Serret1866: Theorem Irreducibility of X^n-a} implies  that \(\gcd(t,\frac {q-1}{\ord(a)})=1\). If \(p \mid t\), then also \(\gcd(p, \frac {q-1}{\ord(a)})=1\) and \(p \cdot \ord(a)\) cannot divide \((q-1)\). \Cref{theorem: irreducibility of X^(tp)-a} completes our proof.
\end{proof}

\Cref{theorem: irreducibility of X^(tp)-a with p mid t} simplifies the check for the irreducibility of the polynomial \(X^{tp}-a\) significantly.  If \(p \mid t\), we do not need to compare the properties of \(t\) and \(\ord(a)\) as specified in \Cref{Serret1866: Theorem Irreducibility of X^n-a} (i) and (ii).

Using  \Cref{proposition: ord(b) s.t. b^p=a} and \Cref{theorem: irreducibility of X^(tp)-a}, the following theorem gives the factorization of \(X^{tp}-a\) for the case that the polynomial is reducible over \(\Fq\) and (\(4\nmid tp\) or \(q \equiv 1 \pmod 4\)).

\begin{proposition}\label{theorem: Factorization of X^(tp)-a reducible case}
	Let \(a\in \Fq^\ast\), \(p\) be a prime such that \(p \mid (q-1)\) and let \(t\in \N\) such that \(X^t-a\) is irreducible over \(\Fq\). Further,  let  \(4\nmid tp\) or \(q \equiv 1 \pmod 4\). Then if \(p\cdot \ord(a)\mid (q-1)\), the factorization of \(X^{tp}-a\) into monic irreducible factors over \(\Fq\) is
	\begin{equation*}
		X^{tp}-a = \prod_{j=0}^{p-1} (X^t-\cyclgen p^j b),
	\end{equation*}
	where \(b\in \Fq\) such that \(b^p=a\). 
	\begin{enumerate}[(i)]
		\item If \(p \nmid \ord(a)\),  then \(b=a^r\)  and \(\ord(b)= \ord(a)\)  for a positive integer \(r\) satisfying \changed{\(r=1\) if \(a=1\) or  \(rp \equiv 1 \pmod {\ord(a)}\) otherwise}. \label{item: X^{tp}-a order if p nmid ord(a)}
		\item If \(p \mid \ord(a)\), \(\ord(b)= p\cdot \ord(a)\). \label{item: X^{tp}-a order if p mid ord(a)}
	\end{enumerate}
	Furthermore, \(\ord(\cyclgen p^j b) = p \cdot \ord(a)\) for all \(1 \leq j \leq p-1\).
\end{proposition}

\begin{proof}
	From  \Cref{proposition: ord(b) s.t. b^p=a} and \Cref{theorem: irreducibility of X^(tp)-a} follows directly that 
	\(X^{tp}-a = \prod_{j=0}^{p-1} (X^t-\cyclgen p^j b)\) and that  statements (\ref{item: X^{tp}-a order if p nmid ord(a)}) and (\ref{item: X^{tp}-a order if p mid ord(a)}) hold. It remains to prove that \(X^t-\cyclgen p^jb\) is irreducible for every \(0\leq j \leq p-1\). Since \(X^t-a\) is irreducible, from \Cref{Serret1866: Theorem Irreducibility of X^n-a} follows that \(\rad(t)\mid \ord(a)\) and \(\gcd(t, \frac{q-1}{\ord(a)})=1\). Consequently,  \(\rad(t)\mid \ord(\cyclgen p^j b)\), because \(\ord(\cyclgen p^j b)\in \{\ord(a), p \cdot \ord(a)\}\). Furthermore, \(\frac {q-1}{\ord(\cyclgen p^j b)}\) divides \(\frac{q-1}{\ord(a)}\) which implies that \(\gcd(t, \frac {q-1}{\ord(\cyclgen{p}^jb)})=1\) and \(X^{t}-\cyclgen p^j b\) is irreducible with \Cref{Serret1866: Theorem Irreducibility of X^n-a}. 
\end{proof}

\Cref{theorem: irreducibility of X^(tp)-a} and  \Cref{theorem: Factorization of X^(tp)-a reducible case} can be applied recursively and we use them for the induction step in the proofs of the next two theorems, \Cref{theorem: factorization X^(tn)-a with rad(n) mid ord(a)} and \Cref{theorem: factorization X^(tn)-a with gcd(n_ord(a))=1}. Note that \Cref{theorem: irreducibility of X^(tp)-a} and  \Cref{theorem: Factorization of X^(tp)-a reducible case} allow us to factor the polynomial \(X^{tp}-a\) simply by examining the order of the element \(a\) and comparing it with \(p\) and \(q-1\). This is why \Cref{theorem: Factorization of X^(tp)-a reducible case} (i) and (ii), which specify the orders of the new appearing coefficients, are important.

Propositions \ref{theorem: factorization X^(tn)-a with rad(n) mid ord(a)} and \ref{theorem: factorization X^(tn)-a with gcd(n_ord(a))=1} specify the factorization of the polynomial \(X^{tn}-a\) for the two cases \(\rad(n)\mid \ord(a)\) and  \(\gcd(n,\ord(a))=1\), where \(X^t-a\) is irreducible over \(\Fq\).

\begin{proposition}
	\label{theorem: factorization X^(tn)-a with rad(n) mid ord(a)}
	Let \(a\in \Fq^\ast\), \(t\in \N\) such that \(X^t-a\) is irreducible over \(\Fq\) and let \(n\in \N\) such that \(\rad(n)\mid \ord(a)\). Further,  let  \(4\nmid tn\) or \(q \equiv 1 \pmod 4\). Set \(d:= \gcd(n, \frac {q-1}{\ord(a)})\). Then the factorization of \(X^{tn}-a\) into monic irreducible factors over \(\Fq\) is
	\begin{equation*}
		\prod_{j=0}^{d-1} (X^{t\cdot \frac nd}- \cyclgen d^jb),
	\end{equation*} 	
	where \(b\in \Fq\) such that \(b^d=a\) and \(\ord(\cyclgen d^jb) = \ord(a)\cdot d\) for all \(0 \leq j \leq d-1\).
\end{proposition}

\begin{proof}
	We use the fact that every positive integer \(n\) can be written as a product of prime factors and prove the statement by induction  on the number of prime factors of \(n\). \underline{Base case:} If \(n=p\) for a prime \(p\) such that \(p \mid q-1\) and \(p \mid \ord(a)\), the statement follows directly from \Cref{theorem: Factorization of X^(tp)-a reducible case} (i).
	
	\underline{Induction hypothesis:} Suppose that the statement of \Cref{theorem: factorization X^(tn)-a with gcd(n_ord(a))=1} holds for \(n, d, b \)  such that \(\rad(n)\mid \ord(a), d= \gcd(n,\frac {q-1} {\ord(a)})\) and \(b^d = a\).
	
	\underline{Induction step \((n\rightarrow np)\):} Let \(p\) be a prime such that \(p\mid \ord(a)\), then
	\begin{equation}\label{eq: factorization X^(tnp)-a for rad(n) mid ord(a)}
		X^{tnp}-a = \prod_{j=0}^{d-1} (X^{t\cdot \frac nd\cdot p}-\cyclgen d^j b).
	\end{equation}
	If  \(dp\) does not divide \(\frac {q-1}{\ord(a)}\), then with \Cref{theorem: irreducibility of X^(tp)-a} the polynomial \((X^{t\cdot \frac nd \cdot p }-\cyclgen d^j b)\) is irreducible over \(\Fq\) for every \(0 \leq j\leq d-1\). Furthermore, the greatest common divisor of \(np\) and \(\frac {q-1}{\ord(a)}\) is also \(d\).  Consequently, \cref{eq: factorization X^(tnp)-a for rad(n) mid ord(a)} is the factorization of \(X^{tnp}-a \) into monic irreducible factors over \(\Fq\) and the statement of \Cref{theorem: factorization X^(tn)-a with rad(n) mid ord(a)} holds.
	
	If \(dp\) divides \(\frac {q-1}{\ord(a)}\), then \(\tilde d := \gcd(np, \frac {q-1}{\ord(a)}) = dp\) and there exists \(\cyclgen{\tilde d}\in \Fq\). Furthermore, for every \(0 \leq j \leq d-1\) holds \(\ord(\cyclgen d^j b) \cdot p = \ord(a)\cdot d \cdot p\mid q-1\) and the polynomial \((X^{t \cdot \frac nd}-\cyclgen d^j b)\) is reducible. \Cref{theorem: Factorization of X^(tp)-a reducible case} yields that there exists an element \(c_j\in \Fq\) such that \(c_j^p = \cyclgen d^j b \) and that the factorization of \((X^{t \cdot \frac nd \cdot p}-\cyclgen d^j b)\) into monic irreducible factors over \(\Fq\) is given by \(\prod_{i=0}^{p-1} (X^{t \cdot \frac nd}- \cyclgen p^i c_j)\), where \(\ord(\cyclgen p^i c_j) = \ord(\cyclgen d^j b)\cdot p = \ord(a)\cdot d \cdot p = \ord(a)\cdot \tilde d\). 
	Thus, the factorization of \(X^{tnp}-a\) into monic irreducible factors over \(\Fq\) is
	\begin{equation}\label{eq: factorization X^{tnp}-a rad(n) mid ord(a) (1)}
		X^{tnp}-a = \prod_{j=0}^{d-1} \prod_{i=0}^{p-1}(X^{t\cdot \frac nd}-\cyclgen{p}^i c_j) = \prod_{j=0}^{d-1} \prod_{i=0}^{p-1} (X^{t\cdot \frac {np} {\tilde d}} - \cyclgen p^i c_j).
	\end{equation} Moreover, \(c:=c_0\) satisfies \(c^{dp} = (c^p)^d = b^d =a\) and the factorization of  \(X^{tnp}-a\) into monic irreducible factors over \(\Fq\) can also  be written as:
	\begin{equation}\label{eq: factorization X^{tnp}-a rad(n) mid ord(a) (2)}
		X^{tnp}-a = \prod_{j=0}^{\tilde d} (X^{t\cdot \frac {np} {\tilde d}}-\cyclgen{\tilde d}^j c).
	\end{equation}
	Then for every \(0 \leq j \leq \tilde d -1\) holds \(\ord(\cyclgen {\tilde d}^jc) = \ord(a)\cdot \tilde d\), because every factor in \cref{eq: factorization X^{tnp}-a rad(n) mid ord(a) (2)} corresponds to a factor in \cref{eq: factorization X^{tnp}-a rad(n) mid ord(a) (1)}.
\end{proof}

As in the proof of \Cref{theorem: factorization X^(tn)-a with rad(n) mid ord(a)} we obtain the factorization of \(X^{tn}-a\) for the case \(\gcd(n, \ord(a))=1\) by induction on the number of prime factors of \(n\).  Additionally, we impose the condition \(\gcd(n,t)=1\) on \(n\) because otherwise the polynomials would not be reducible in every induction step (see \Cref{theorem: irreducibility of X^(tp)-a with p mid t}).

\begin{proposition}
	\label{theorem: factorization X^(tn)-a with gcd(n_ord(a))=1}
	Let \(a\in \Fq^\ast\), \(t\in \N\) such that \(X^t-a\) is irreducible over \(\Fq\) and let \(n\in \N\) such that \(\rad(n)\mid (q-1)\) and \(\gcd(n,\ord(a)\cdot t)=1\). Further,  let  \(4\nmid tn\) or \(q \equiv 1 \pmod 4\). Set \(d:= \gcd(n, \frac {q-1}{\ord(a)}) = \gcd(n, q-1)\). Then the factorization of \(X^{tn}-a\) into monic irreducible factors over \(\Fq\) is
	\begin{equation}\label{eq: factorization of X^(tn)-a with gcd(n_ord(a))=1}
		\prod_{v \mid \frac n{d} } \prod_{\substack{j=0\\\gcd(j,v)=1}}^{d-1} (X^{tv}-\cyclgen {d}^{j}a^{rv}),
	\end{equation}
	where \(r\in \N\) satisfies \changed{\( r=1\) if \(a=1\) or  \(rn \equiv 1\pmod {\ord(a)}\) otherwise}  and for all applicable \(v\) and \(j\) holds 
	\begin{enumerate}[(i)]
		\item \((\cyclgen {d}^j a^{rv})^{\frac nv} = a\),
		\item \(\ord( \cyclgen {d}^j a^{rv}) = \ord(a)\cdot \frac {d} {\gcd(j,d)}\).
	\end{enumerate}
\end{proposition}

\begin{proof}
	\underline{Base case:} If \(n=p\), then the statement follows directly from \Cref{theorem: Factorization of X^(tp)-a reducible case} (ii).
	
	\underline{Induction hypothesis:} Suppose that the statement of \Cref{theorem: factorization X^(tn)-a with gcd(n_ord(a))=1} holds for \(n, d, r\) such that \(\rad(n) \mid (q-1), \gcd(n,\ord(a)\cdot t)=1, d = \gcd(n, q-1)\) and either \(a=1\) and \(r=1\) or  \(rn\equiv 1 \pmod {\ord(a)}\). 
	
	\underline{Induction step (\(n\rightarrow np\)):} Let \(p\) be a prime such that \(p \mid (q-1)\) and \(p \nmid \ord(a)\cdot t\). \changed{If \(a=1\) then set \(\tilde r=1\) so that \((a^{\tilde r})^p = a\).} Otherwise there exists a positive integer \( r'\) such that \( r' p \equiv 1 \pmod {\ord(a)}\). Since \(\ord(a^r)=\ord(a)\), this implies that \(a^{r r' p}=(a^r)^{ r'p} = a^r\). We set \(\tilde r:= r r'\) and together with \cref{eq: factorization of X^(tn)-a with gcd(n_ord(a))=1} we obtain:
	\begin{equation}\label{eq: proof  X^(tnp)-a for gcd(n_ord(a))=1}
		X^{tnp}-a = \prod_{v\mid \frac n {d}} \prod_{\substack{j=0\\\gcd(j,v)=1}}^{d-1} (X^{tvp}- \cyclgen {d}^j a^{\tilde rpv}),
	\end{equation}
	where \(\tilde r\) satisfies \changed{\(r=1\) if \(a=1\) or  \(\tilde r\cdot np = rn \cdot r'p \equiv 1 \pmod {\ord(a)}\) otherwise}. Furthermore, we have \(\ord(a^{\tilde r pv}) = \ord(a)\), because \((a^{\tilde r pv})^{\frac nv} = a\). It remains to discuss which of the factors \((X^{tvp}-\cyclgen{d}^j a^{\tilde rvp})\) are reducible. For this we need to discern the two cases \(\gcd(np,q-1)= d\) and \(\gcd(np, q-1)= d\cdot p\).
	
	\underline{Case \(\gcd(np,q-1)=d\):} If \(\gcd(np,q-1)=d\), then \(p \mid n\) but \(dp\nmid q-1\) and also \(\ord(a)\cdot d\cdot p \nmid q-1\). Let \(v\mid \frac n {d}\) and \(0 \leq j \leq d-1\) such that \(\gcd(j,v)=1\). Then if \(p \mid v\), the polynomial \((X^{tvp}-\cyclgen {d}^ja^{\tilde rvp})\) is irreducible over \(\Fq\) with \Cref{theorem: irreducibility of X^(tp)-a with p mid t}. Thus, the following expression is a factorization into irreducible factors over \(\Fq\), where we set \(\tilde v:=vp\):
	\begin{equation}\label{eq: gcd(np_q-1)=d p^2 mid v}
		\prod_{\substack{v\mid \frac n {d}\\ p \mid v}} \prod_{\substack{j=0\\ \gcd(j,v)=1}}^{d-1} (X^{tvp}-\cyclgen {d}^j a^{\tilde rvp}) = \prod_{\substack{\tilde v\mid \frac{np}{d}\\ p^2\mid \tilde v}} \prod_{\substack{j=0\\\gcd(j,\tilde v)=1}}^{d-1} (X^{t\tilde v}-\cyclgen{d}^j a^{\tilde r\tilde v}),
	\end{equation}
	where \(\gcd(j,v)=1\) if and only if \(\gcd(j,\tilde v)=\gcd(j,vp)=1\), because \(p \mid v\). If \(p \nmid v\), then with \Cref{theorem: irreducibility of X^(tp)-a} the polynomial \((X^{tvp}-\cyclgen {d}^ja^{\tilde rvp})\) is irreducible if and only if \(\ord(\cyclgen {d}^j a^{\tilde rvp})  \cdot p = \ord(a) \cdot \frac {d}{\gcd(j,d)} \cdot p\) does not divide \((q-1)\). This is the case if and only if \(p \nmid j\), because \(\ord(a)\cdot d\) divides \(( q-1)\) but  \(\ord(a)\cdot d\cdot p\) does not. Consequently, the following expression is a factorization into irreducible factors over \(\Fq\), where we set \(\tilde v:= vp\):
	\begin{equation}\label{eq: gcd(np_q-1)=d p mid v}
		\prod_{\substack{v\mid \frac n {d}\\ p \nmid v}} \prod_{\substack{j=0\\\gcd(j,v)=1\\p \nmid j}}^{d-1} (X^{tvp}-\cyclgen {d}^j a^{\tilde rvp}) = \prod_{\substack{\tilde v\mid \frac {np}{d}\\ \nu_p(\tilde v)=1}} \prod_{\substack{j=0\\ \gcd(j,\tilde v)=1}}^{d-1} (X^{t\tilde v}-\cyclgen {d}^j a^{\tilde r\tilde v}).
	\end{equation}
	Let \(0 \leq j \leq d-1\) such that \(p \mid j\). Then the polynomial \((X^{tvp}-\cyclgen{d}^j a^{\tilde rvp})\) is reducible and \(\cyclgen {d}^{\frac jp} a^{\tilde r v} \) is an element in \(\Fq\) satisfying \((\cyclgen {d}^{\frac jp} a^{\tilde r v})^p = \cyclgen {d}^j a^{\tilde r vp}\).  Let \(\cyclgen p = \cyclgen {d}^{\frac {d}p}\), then from \Cref{theorem: Factorization of X^(tp)-a reducible case} follows that its factorization into irreducible factors over \(\Fq\) is
	\begin{equation*}
		\prod_{i=0}^{p-1} (X^{tv}-\cyclgen p^i \cyclgen {d}^{\frac jp} a^{\tilde r v}) = \prod_{i=0}^{p-1} (X^{tv}-\cyclgen{d}^{\frac {d}p \cdot i + \frac jp } a^{\tilde r v}).
	\end{equation*}
	With the fact that \(\{0 \leq j \leq d-1: p \mid j\}= \{p\tilde j: 0 \leq \tilde j \leq \frac {d} p -1\}\), the following expression is a factorization into irreducible factors over \(\Fq\):
	\begin{equation}\label{eq: gcd(np_q-1)=d p nmid v}
		\prod_{\substack{v\mid \frac n {d}\\ p \nmid v}} \prod_{\tilde j =0}^{\frac {d}p -1} \prod_{i=0}^{p-1} (X^{tv}-\cyclgen{d}^{\frac {d}p \cdot i+\tilde j} a^{\tilde r v}) = \prod_{\substack{v\mid \frac {np} {d}\\ p \nmid v}} \prod_{j=0}^{d-1} (X^{tv}-\cyclgen{d}^j a^{\tilde rv}),
	\end{equation}
	where \(\ord(a^{\tilde r v})=\ord(a)\), because \((a^{\tilde rv})^{p \frac nv}=a\). Equations (\ref{eq: gcd(np_q-1)=d p^2 mid v}), (\ref{eq: gcd(np_q-1)=d p mid v}) and (\ref{eq: gcd(np_q-1)=d p nmid v}) combined yield that if \(\gcd(np,q-1)=d\) the factorization of \(X^{tnp}-a\) into irreducible factors over \(\Fq\) is given by:
	\begin{equation*}
		\prod_{v\mid \frac {np}{d}} \prod_{\substack{j=0\\\gcd(j,v)=1}}^{d-1} (X^{tv}-\cyclgen{d}^j a^{\tilde r v}),
	\end{equation*}
	where for all applicable \(v\) and \(j\) holds:
	\begin{enumerate}[(i)]
		\item \(( \cyclgen{d}^j a^{\tilde r v})^{\frac {np} v} = \cyclgen{d}^{j \frac {np}v} a^{\tilde r np } = a\), because \(v\mid \frac {np} {d}\) implies \(d\mid \frac {np} v\),
		\item \(\ord(\cyclgen{d}^j a^{\tilde r v}) = \frac {d}{\gcd(j,d)} \cdot \ord(a) \), because \(\gcd(d, \ord(a))=1\).
	\end{enumerate}
	
	\underline{Case \(\gcd(np,q-1)=d\cdot p\):} If \(\gcd(np, q-1)= dp\), then \(d\cdot p \mid (q-1)\) and there exists \(\cyclgen{dp}\in \Fq\) such that \(\cyclgen{dp}^{p}=\cyclgen{d}\). Additionally, we set \(\cyclgen p := \cyclgen{dp}^{d}\).  We consider the polynomial \((X^{tvp}-\cyclgen{d}^j a^{\tilde r vp})\) from \cref{eq: proof  X^(tnp)-a for gcd(n_ord(a))=1}, where \(v \mid \frac {n} {d}\) and \(0 \leq j \leq d-1\) such that \(\gcd(j,v)=1\). Since \(\gcd(d\cdot p, \ord(a))=1\) and \(dp\mid (q-1)\), the product \(\ord(a)\cdot d\cdot p\) divides \((q-1)\). Consequently, also  \(\ord(\cyclgen{d}^j a^{\tilde rvp}) \cdot p\) divides  \((q-1)\) and with \Cref{theorem: Factorization of X^(tp)-a reducible case}  the factorization of \((X^{tvp}- \cyclgen{d}^j a^{\tilde r vp})\) into monic irreducible factors over \(\Fq\) is
	\begin{align*}
		\prod_{i=0}^{p-1} (X^{tv}-\cyclgen{p}^i \cyclgen{dp}^j a^{\tilde rv}) = \prod_{i=0}^{p-1} (X^{tv}-\cyclgen{dp}^{d\cdot i + j} a^{\tilde r v}),
	\end{align*}
	because \(\cyclgen{dp}^j a^{\tilde rv}\) is an element of \(\Fq\) satisfiying \((\cyclgen{dp}^j a^{\tilde rv})^p = \cyclgen{d}^j a^{\tilde r vp}\). 
	
	Note that \(\{di+j: 0 \leq i \leq p-1, 0 \leq j \leq d-1\} = \{0 \leq j \leq dp-1\}\). Therefore, the factorization of \(X^{tnp}-a\) into monic irreducible factors over \(\Fq\) is 
	\begin{equation*}
		\prod_{v\mid \frac {n} {d}} \prod_{\substack{j=0\\ \gcd(j,v)=1}}^{d-1} \prod_{i=0}^{p-1} (X^{tv}-\cyclgen{dp}^{di+j} a^{\tilde r v}) = \prod_{v \mid \frac {np}{dp}} \prod_{\substack{j=0\\\gcd(j,v)=1}}^{dp-1} (X^{tv}- \cyclgen{dp}^j a^{\tilde r v}),
	\end{equation*}
	where \(\gcd(j,v)=1\) if and only if \(\gcd(di+j, v)=1\), because every prime \(w\) that divides \(v\) must divide  \(n\). Since \(\rad(n)\mid q-1\), \(w\) also divides \(d\). Furthermore, for all applicable \(v\) and \(j\), we have 
	\begin{enumerate}[(i)]
		\item \((\cyclgen{dp}^j a^{\tilde r v})^{\frac {np} v} = (\cyclgen{dp}^j a^{\tilde r v})^{p \cdot  \frac {n} v}\), because \(p\) does not divide \(v\). Indeed, suppose that \(p \mid v\), then with \Cref{theorem: irreducibility of X^(tp)-a with p mid t} the polynomial \((X^{tvp}-\cyclgen{d}^j a^{\tilde r vp})\) would not have been reducible. Thus, \((\cyclgen{dp}^j a^{\tilde r v})^{\frac {np} v} = (\cyclgen{d}^j a^{\tilde r vp})^{\frac nv} = a\) by induction hypothesis. 
		\item \(\ord(\cyclgen{dp}^j a^{\tilde rv}) = \frac {dp}{\gcd(j, dp)}\cdot \ord(a)\), because \(\ord(a^{\tilde rv})=\ord(a)\) and  \(\gcd(dp, \ord(a))=1\).
	\end{enumerate}
\end{proof}

The next theorem combines  \Cref{theorem: factorization X^(tn)-a with rad(n) mid ord(a)} and \Cref{theorem: factorization X^(tn)-a with gcd(n_ord(a))=1}. It gives the complete  factorization of \(X^n-a\) for the case \(\rad(n)\mid (q-1)\) and (\(4\nmid n\) or \(q\equiv 1 \pmod 4\)). \changed{In particular, it is an extension and a combination  of \cite[Theorem 6 (1), Theorem 8 (1), Theorem 9 (1)]{WuYue2018constacyclic}.}

\begin{theorem}\label{theorem: factorization of X^n-a for rad(n) mid q-1_4 nmid n or q = 1 mod 4}
	Let \(a\in \Fq^\ast\) and  \(n\in\N\) such that \(\rad(n)\mid q-1\) and \((4 \nmid n\) or \(q\equiv 1 \pmod 4)\). We write \(n=n_1 \cdot n_2,\) where \(\rad(n_1)\mid \ord(a)\) and \(\gcd(n_2, \ord(a))=1\). Further, we set \(d_1 := \gcd(n_1, \frac {q-1}{\ord(a)})\) and \(d_2:= \gcd(n_2, q-1)\).  Then there exists \(b\in \Fq\) such that \(b^{d_1} = a\) and the factorization of \(X^n-a\) into monic irreducible factors over \(\Fq\) is 
	\begin{equation*}
		\prod_{j=0}^{d_1-1} \prod_{v \mid \frac {n_2}{d_2}} \prod_{\substack{i=0\\\gcd(i,v)=1}} ^{d_2-1}   (X^{\frac {n_1}{d_1} \cdot v}- \cyclgen{d_2}^i (\cyclgen{d_1}^{j} b)^{rv}),
	\end{equation*}
	where \(r \in \N\) satisfies \changed{\(r=1\) if \(a=1\) or \(rn_2 \equiv 1 \pmod {\ord(a)\cdot d_1}\) otherwise } and for all applicable \(j,v,i\) holds:
	\begin{enumerate}[(i)]
		\item \((\cyclgen{d_2}^i (\cyclgen{d_1}^{j} b)^{rv}) ^{\frac {n_2} {v} \cdot d_1}=a\),
		\item \(\ord(\cyclgen{d_2}^i (\cyclgen{d_1}^{j} b)^{rv}) = \ord(a)\cdot d_1  \cdot \frac {d_2}{\gcd(i,d_2)}\).
	\end{enumerate}
\end{theorem}

\begin{proof}
	From \Cref{theorem: factorization X^(tn)-a with rad(n) mid ord(a)} follows that the factorization of \(X^{n_1}-a\)  is  \(\prod_{j=0}^{d_1-1} (X^{\frac {n_1}{d_1}}-\cyclgen{d_1}^{j} b) \), where \(b\in\Fq\) such that \(b^{d_1}=a\) and  \(\ord(\cyclgen{d_1}^{j} b)=\ord(a)\cdot d_1\) for all \(0 \leq j \leq d_1-1\). Since \(\gcd(n_2, \ord(a)\cdot d_1\cdot \frac {n_1}{d_1})=1\), application of \Cref{theorem: factorization X^(tn)-a with gcd(n_ord(a))=1} yields the factorization of \((X^{\frac {n_1}{d_1}}- \cyclgen{d_1}^{j}b)\):
	\begin{equation*}
		\prod_{v \mid \frac {n_2}{d_2}} \prod_{\substack{i=0\\\gcd(i,v)=1}}^{d_2-1}  (X^{\frac {n_1}{d_1} \cdot v}- \cyclgen{d_2}^i (\cyclgen{d_1}^j b)^{rv}),
	\end{equation*}
	where \((\cyclgen{d_2}^{i} (\cyclgen{d_1}^{j}b)^{rv})^{\frac {n_2}v} = (\cyclgen{d_1}^j b )\) and \(\ord(\cyclgen{d_2}^{i} (\cyclgen{d_1}^{j}b)^{rv}) = \ord(\cyclgen{d_1}^{j}b) \cdot \frac {d_2}{\gcd(i, d_2)}= \ord(a)\cdot d_1 \cdot \frac {d_2}{\gcd(i,d_2)}\). 
\end{proof}

\begin{remark}
	With the setup of \Cref{theorem: factorization of X^n-a for rad(n) mid q-1_4 nmid n or q = 1 mod 4}, there exists a \((d_1d_2)\)-th primitive root of unity in \(\Fq\), because \(\gcd(d_1,d_2)=1\). Then we can set \(\cyclgen{d_1} = \cyclgen{(d_1d_2)}^{d_2}\) and \(\cyclgen{d_2} = \cyclgen{(d_1d_2)}^{d_1}\) and rewrite 
	\(\cyclgen{d_2}^i (\cyclgen{d_1}^{j} b)^{rv}= \cyclgen{(d_1d_2)}^{d_1\cdot i + d_2 \cdot jrv} \cdot b^{rv}\). However, this notation seemed less readable than the expression in \Cref{theorem: factorization of X^n-a for rad(n) mid q-1_4 nmid n or q = 1 mod 4}.
\end{remark}

\subsection{The factorization of $X^n-a$ for positive integers $n$\\ such that $\gcd(n,q)=1$}

\label{subsection: factorization X^n-a gcd(n q)=1}

\Cref{theorem: factorization of X^n-a for rad(n) mid q-1_4 nmid n or q = 1 mod 4} combined with the following corollary, \Cref{theorem: corollary KK11 lemma 1}, allows us to determine the factorization of \(X^n-a\) for any positive integer \(n\) such that \(\gcd(n,q)=1\) in the following way: Let \(\gcd(n,q)=1\), then \(w=\ord_{\rad(n)}(q)\) is the smallest positive integer such that  \(\rad(n)\mid q^w-1\). If \(4\mid n\) and \(q^w\equiv 3 \pmod 4\), then the extension field \(\Fqto{2w}\) satisfies \(q^{2w} \equiv 1 \pmod 4\) and also \(\rad(n)\mid q^{2w}-1\). Thus, \Cref{theorem: factorization of X^n-a for rad(n) mid q-1_4 nmid n or q = 1 mod 4} yields the factorization of \(X^n-a\) into monic irreducible factors over \(\Fqto{s}\), where \(s=w\) or \(s={2w}\) respectively.  The factorization of \(X^n-a\) over \(\Fq\) can be derived from the factorization of \(X^n-a\) over \(\Fqto s\) with the following  corollary of  \Cref{KK11: Lemma 1}. 

\begin{corollary}\label{theorem: corollary KK11 lemma 1}
	Let \(g\) be a monic polynomial over \(\Fq\) and \(\prod_{R\in \mathcal R} R\) its factorization into monic irreducible factors over \(\Fqto{w}\) for \(w\geq 1\). Then the factorization of \(g\) into monic irreducible polynomials  over \(\Fq\) is given by:
	\[\prod_{R \in \mathcal R/\sim} \spin q {R},\]
	where \(R\sim \tilde R\) if and only if \(\tilde R = \frobenius{R}{j}\) for a positive integer \(j\). 
\end{corollary}

The authors of \cite{WuYueFan2018, WuYue2021} use the method described above to obtain their results on the factorization of \(X^n-1\) such that \(\rad(n)\) is a divisor of \(q^w-1\) or \(q^{vw}-1\), where \(v,w\) are primes. In \cite{Brochero-MartinezReisSilva-Jesus2019} the authors use this method to factor \(f(X^n)\) such that \(\gcd(n, \ord(f)\cdot \deg(f))=1\) and \(\rad(n)\mid q^w-1\) for a prime \(w\). 

\changed{For the formulation and the proof of \Cref{theorem: factorization X^n-a for gcd(n_q)=1} we need the following four results. The following proposition  is a consequence of \Cref{proposition: ord(b) s.t. b^p=a} and  states that for every positive integer \(t\) such that \(X^t-a\) is irreducible, every root of \(X^t-a\) has order \(t\cdot \ord(a)\).
	
	\begin{nproposition}\label{theorem: ord(X^t-a)}
		Let \(a\in \Fq^\ast\) and  \(t\) be a positive integer such that \(X^t-a\) is irreducible over \(\Fq\). Then the order of \(X^t-a\) equals \(t \cdot \ord(a)\).
	\end{nproposition}
	
	\begin{proof}
		We prove the statement by induction on the number of prime factors of \(t\). For \(t=1\) the statement is obviously true. Next, let  \(t\) be a positive integer such that \(X^t-a\) is irreducible and \(\ord(X^t-a) = t \cdot \ord(a)\).  Further, let \(p\) be a prime   such that the binomial \(X^{tp}-a\) is irreducible.  Let \(\gamma \in \Fqto {tp}\) be a root of \(X^{tp}-a\). Then   \(\beta := \gamma ^p\) is a root of \(X^t-a\) and satisfies \(\ord(\beta) = t \cdot \ord(a)\).  Thus, \(\gamma\) is a root of the polynomial \(X^p-\beta\in \Fqto{t}[X]\). Since \(X^{tp}-a\) is irreducible, we have \(\gcd(p,q)=1\) and with \Cref{proposition: ord(b) s.t. b^p=a} follows that \(\ord(\gamma)\in \{\ord(\beta), p \cdot \ord(\beta)\}\). Suppose that \(\ord(\gamma) = \ord(\beta)\). Then \(\gamma\) is a proper element of \(\Fqto t\) and the minimal polynomial of \(\gamma\) over  \(\Fq\) has degree \(t\). However, since \(X^{tp}-a\) is irreducible and \(\gamma \) is a root of \(X^{tp}-a\), this is the minimal polynomial of \(\gamma\) over \(\Fq\). A contradiction and our assumption must have been wrong. Consequently, \(\ord(\gamma) = p \cdot \ord(\beta)\).
\end{proof}

The following fact is well known. Its proof shows the mechanics behind the case \(\rad(n)\mid q-1\), \(4\mid n\) and \(q\equiv 3 \pmod 4\):

\begin{nfact}\label{fact: gcd(n q^2-1) for q=3 mod 4}
	Let \(q\equiv 3 \pmod 4\)  and \(n\) be a positive integer such that \(\rad(n)\mid q-1\) and  \(4\mid n\). 
	\begin{enumerate}[(i)]
		\item Then \(\gcd(n,q^2-1) = \gcd(n,q-1) \cdot  2^{1+\min\{\nu_2(\frac n4), \nu_2(\frac {q+1}2)\}}.\)\label{item: Fact q^2-1 q equiv 3 mod 4 d_2}
		\item Let \(e\) be a divisor of \(q-1\). Then \( {\gcd(n,\frac {q^2-1}{e})} = {\gcd(n,\frac {q-1}e)} \cdot  2^{1+\min\{\nu_2(\frac n4), \nu_2(\frac {q+1}2)\}}\) \label{item: Fact q^2-1 q equiv 3 mod 4 d_1}
	\end{enumerate}
\end{nfact}
\begin{proof}
	\begin{enumerate}[(i)]
		\item Since \(q\equiv 3 \pmod 4\), we have \(2\mid q-1\) but \(4\nmid q-1\). Furthermore, \(q^2-1 = (q-1)\cdot (q+1)\) and \(2\mid (q+1)\), which implies that \(q^2-1\) is divisible by \(4\). Then \(4\) also divides \(\gcd(n,q^2-1)\), because \(4\mid n\).  Note that \(q^2-1\) and \(q-1\) do not have any other common prime factors apart from \(2\). Consequently, \(\gcd(n,q^2-1) = 2^{\min\{\frac n2, \nu_2(q+1)\}} \cdot \gcd(n,q-1)\). 
		\item Since \(e\) divides \(q-1\), we have \(\nu_2(e) \leq 1\). If \(2 \nmid e\), then \(4\mid n\) and (\ref{item: Fact q^2-1 q equiv 3 mod 4 d_2}) imply that  \(\nu_2(\gcd(n, \frac {q-1}e))=1\) and \(\nu_2(\gcd(n,\frac {q^2-1}e))=2+\min \{\nu_2(\frac n4), \nu_2(\frac {q+1}2)\}\). If \(2\mid e\), then \(\nu_2(\gcd(n,\frac {q-1}e))=0\) and \(\nu_2(\gcd(n, \frac {q-1}e))= 1+\min \{\nu_2(\frac n4), \nu_2(\frac {q+1}2)\}\). In both cases, the statement is true, because there are no common prime factors of \(q^2-1\) and \(q-1\)  other than \(2\).
	\end{enumerate}
	
\end{proof}

The following lemma specifies the order and the degree of the irreducible factors of \(X^n-a\) for any positive integer \(n\) such that  \(\rad(n)\mid \ord(a)\), a special case of \(\rad(n)\mid (q-1)\). We use it in the proof of \Cref{theorem: factorization X^n-a for gcd(n_q)=1} to determine one of the parameters (namely \(s_1\)) explicitly.  Note that the information on the degrees of the irreducible factors could be derived from \cite[Theorem 8]{WuYue2018constacyclic}. However, it would have been more complicated to state the result here and show that it is applicable for our case than to give a direct proof.

\begin{nlemma}\label{lemma: deg of irred factors of X^n-a for rad(n) mid ord(a)}
	Let \(a\in \Fq^\ast\) and  \(n\in \N\) such that \(\rad(n)\mid \ord(a)\). Set \(\frobenius {d_1}s := \gcd(n,\frac{q^s-1}{\ord(a)})\) for \(s \in \{1,2\}\). Then every root of the polynomial \(X^n-a\) has order \(\ord(a)\cdot n\) and   is a proper element of \(\Fqto{k}\), where \[k = \begin{cases}
		\frac {n }{\frobenius d 1} & \text{ if } 4\nmid n \text{ or } q\equiv 1 \pmod 4,\\
		\frac {2n}{ \frobenius d2}  & \text{ otherwise.}
	\end{cases}\]
\end{nlemma}

\begin{proof} It suffices to determine the order and the degree of the irreducible factors of \(X^n-a\). We set \(s:= 1\) if \(4\nmid n\) or \(q\equiv 1 \pmod 4\), otherwise we set \(s:=2\). Then the factorization of \(X^n-a\) over \(\Fqto s\) is given by \Cref{theorem: factorization X^(tn)-a with rad(n) mid ord(a)} for \(t=1\):
	\begin{equation}\label{eq: factorization X^n-a rad(n) mid ord(a)}
		\prod_{j=0}^{\frobenius ds-1} (X^{t\cdot \frac n{
				\frobenius d s}}-\cyclgen{\frobenius ds}^jb),
	\end{equation}
	where \(b\in \Fqto s\) such that \(b^{\frobenius ds}=a\) and  \(\ord(\cyclgen{\frobenius {d} s}^jb)=\ord(a)\cdot \frobenius d s\) for every \(0 \leq j \leq \frobenius ds-1\).  \Cref{theorem: ord(X^t-a)} implies that for every \(0 \leq j \leq \frobenius ds-1\) the order of \(X^{\frac{n} {\frobenius d s}}-\cyclgen{\frobenius ds}^jb\) is \(\ord(a)\cdot n\). Consequently, the order of every root of the polynomial \(X^n-a\) equals \(\ord(a)\cdot n\). 
	
	If \(4\nmid n\) or \(q\equiv 1 \pmod 4\), then (\ref{eq: factorization X^n-a rad(n) mid ord(a)}) is the factorization of \(X^n-a\) over \(\Fq\) and every irreducible factor has degree \(\frac n {\frobenius d1}\).
	
	If \(4\mid n\) and \(q\equiv 3 \pmod 4\), then (\ref{eq: factorization X^n-a rad(n) mid ord(a)}) is the factorization of \(X^n-a\) over \(\Fqto 2\) and with \Cref{theorem: corollary KK11 lemma 1}, the factorization of \(X^n-a\) over \(\Fq\) is given by \(\prod_{j \in \mathcal J} \spin q {X^{\frac{n}{\frobenius d 2}}-\cyclgen{\frobenius d 2}^jb}\), where \(\mathcal J\) is a representative system of \(\{0, \ldots, \frobenius d2-1\}/\sim\). The equivalence relation \(\sim\) is defined as \(j \sim \tilde j \) if and only \(\cyclgen{\frobenius d 2}^{\tilde j}b = (\cyclgen{\frobenius d2}^jb)^{q^m}\) for an integer \(m \in \{0,1\}\). For every \(j \in \mathcal J\) the order of \(\cyclgen{\frobenius d 2}^{j}b\) equals \(\ord(a)\cdot \frobenius d2\). Since \(\rad(n)\mid q-1\), \(4\mid n\) and \(q\equiv 3 \pmod 4\), \Cref{fact: gcd(n q^2-1) for q=3 mod 4} (\ref{item: Fact q^2-1 q equiv 3 mod 4 d_1}) implies that \(2 \mid \frac {\frobenius d2}{\frobenius d1}\) and \(\frobenius d2\cdot \ord(a)\) does not divide \(q-1\). Consequently, for all \(j\in \mathcal J\) holds \(\coeffdeg q {X^{\frac n {\frobenius d2}}-\cyclgen{\frobenius d2}^jb} = 2\) and every irreducible factor of \(X^n-a\) over \(\Fq\) has degree \(2 \cdot \frac n {\frobenius d2}\).	
\end{proof}
}

In the proof of \Cref{theorem: factorization X^n-a for gcd(n_q)=1} we use the following fact to split the equivalence relation \(\sim\) from \Cref{theorem: corollary KK11 lemma 1} on the irreducible factors of \(X^n-a\) in the extension field \(\Fqto s\) into two separate equivalence relations. This simplifies finding the representative system \(\mathcal R/\sim\) significantly.

\begin{fact}\label{lemma: a^m=b^m iff a=b}
	Let \((G,\cdot )\) be a finite abelian group and  \(a,b\in G\).
	\begin{enumerate}[(i)]
		\item Let \(\gcd(\ord(a), \ord(b))=1\) and \(m_1, m_2,n_1,n_2\in \N\). Then \(a^{m_1}b^{n_1}=a^{m_2}b^{n_2}\) if and only if \(a^{m_1}=a^{m_2}\) and \(b^{n_1}=b^{n_2}\). \label{item: Fact a b coprime order} 
		\item Let \(m\in \N\) such that \(\gcd(m, \ord(a)\cdot \ord(b))=1\). Then \(a^m=b^m\) if and only if \(a=b\).\label{item: Fact a^m b^m}
	\end{enumerate} 
\end{fact}

\begin{proof}
	\begin{enumerate}[(i)]
		\item The equation \(a^{m_1}b^{n_1}=a^{m_2}b^{n_2}\) is equivalent to \(a^{m_1-m_2} = b^{n_2-n_1}\). Since \(\ord(a)\) and \(\ord(b)\) are coprime, this equation holds if and only if \(a^{m_1-m_2} = 1 = b^{n_2-n_1}\) and the statement holds. 
		\item If \(a=b\), then obviously \(a^m=b^m\).  If \(a^m=b^m\), then \(1 = (\frac ab)^m\) and therefore \(\ord(\frac ab)\mid m\). Since \(\ord(b^{-1})=\ord(b)\), we know that \((\frac ab)^{\ord(a)\cdot \ord(b)}=1\) and thus \(\ord(\frac ab)\) divides \( \ord(a)\cdot \ord(b)\). Then \(\ord(\frac ab) =1\), because \(\gcd(m, \ord(a)\cdot \ord(b))=1\). 
	\end{enumerate}
	
\end{proof}

Let \(\mathcal R\) and \(\sim\) be as in \Cref{theorem: corollary KK11 lemma 1} and suppose that the irreducible factor \(R_i\) is of the form \(X^t-\cyclgen d^i\) for fixed positive integers \(t\) and \(d\) such that \(\gcd(d,q)=1\). Then \(R_i\sim R_{\tilde i}\) if and only if \(\cyclgen d^i = (\cyclgen{d}^{\tilde i})^{q^j}\) for a positive integer \(j\). This condition is equivalent to \(\tilde i \equiv i \cdot q^j \pmod d\).  We introduce the \textit{\(q\)-cyclotomic coset  modulo \(d\) containing the positive integer \(i\)}, which is defined as 
\[\cyclocoset q d  i:=\{ i \cdot q^j \pmod d: j\geq 0\}.\]
Then \(R_i\sim R_{\tilde i}\) if and only if \(\tilde i \in \cyclocoset{q}{d}{i}\). 
Let \(\cyclorepsystem q d\) denote a complete set of representatives of the \(q\)-cyclotomic cosets modulo \(d\). With this definition we can formulate our main result -  one closed formula for the factorization of \(X^n-a\) for any positive integer \(n\) such that \(\gcd(n,q)=1\) and any element \(a\) of \( \Fq^\ast\).

\changed{	\begin{tcolorbox}[breakable]
		\begin{theorem}\label{theorem: factorization X^n-a for gcd(n_q)=1}
			Let \(a\in \Fq^\ast\) and  \(n\in\N\) such that \(\gcd(n,	q)=1\) and  \(n=n_1 \cdot n_2,\) where \(\rad(n_1)\mid \ord(a)\) and \(\gcd(n_2, \ord(a))=1\). Let \(w:=\ord_{\rad(n)}(q)\) and set  \(s:=w\) if \(4\nmid n\) or \(q^w\equiv 1 \pmod 4\), else set \(s:=2w\). For all positive integers \(t\)   we set \(\frobenius{d_1} t := \gcd(n_1, \frac {q^t-1}{\ord(a)})\) and \(\frobenius{d_2} t:= \gcd(n_2, q^t-1)\). If \(4 \nmid \frobenius {d_1} s \) or \(q\equiv 1 \pmod 4\), then \(s_1:=\frac {\frobenius{d_1}s} {\frobenius {d_1}1}\), otherwise  \(s_1:=  \frac {2{\frobenius{d_1}s}} {\frobenius {d_1}2} \). For every \(i \in \cyclorepsystem{q}{\frobenius {d_2}s}\) we set \(t_i :=  \min\{t\in \N: \frac{\frobenius{d_2}s}{\frobenius{d_2}t}\mid i\}\)  and \(c_i:= \lcm(t_i,s_1)\). 
			
			Then there exists \(b\in \Fqto{s}\) such that \(b^{\frobenius{d_1}{s}} = a\) and the factorization of \(X^n-a\) over \(\Fq\) is 
			\begin{equation*}
				\prod_{j\in \{0,\ldots, \frobenius{d_1}{s}-1\}/\sim } \; \prod_{v\mid \frac {n_2}{\frobenius{d_2}{s}}} \; \prod_{\substack{i \in \cyclorepsystem{q}{\frobenius{d_2}{s}}\\ \gcd(i,v)=1}} \prod_{m=0}^{\gcd(t_i,s_1)-1}  S_{(j,v,i,m)},\; 
			\end{equation*} 
			where for all applicable \((j,v,i,m)\) the monic irreducible polynomial  \(S_{(j,v,i,m)}\) of degree \(\frac {n_1}{\frobenius{d_1}s} \cdot v \cdot c_i\) and order \(\ord(a) \cdot n_1 \cdot v\cdot   \frac {\frobenius{d_2}{s}}{\gcd(i, \frobenius{d_2}{s})}\) is defined as  \begin{equation*}
				S_{(j,v,i,m)}=\sum_{l=0}^{c_{i}} X^{\frac {n_1}{\frobenius {d_1} s} \cdot v \cdot l}\cdot   (-1)^{c_i-l} \sum_{\substack{\mathcal U \subseteq \{0, \ldots, c_{i}-1\}\\|\mathcal U|=c_i-l}} \prod_{u\in \mathcal U}  (\cyclgen{\frobenius {d_2} s}^{i\cdot q^m} (\cyclgen{\frobenius{d_1} s}^{j} b)^{rv})^{q^u} ,
			\end{equation*} where \(r\) is a  positive integer satisfying \changed{\(r=1\) if \(a=1\) or  \(rn_2 \equiv 1 \pmod{\ord(a)\cdot \frobenius{d_1}{s}}\) otherwise}. 	Furthermore, for all \(j,\tilde j\in \{0 , \ldots, \frobenius {d_1} s-1\}\) holds \(j\sim \tilde j\) if and only  if \(\cyclgen{\frobenius{d_1} s}^{\tilde j} = \cyclgen{\frobenius{d_1}{s}}^{j\cdot q^m}b^{q^m-1}\) for an integer \(0\leq m\leq s_1-1\).
		\end{theorem}
	\end{tcolorbox}

\begin{proof}
	Since \(\rad(n)\mid q^s-1\) and \(4\nmid n\) or \(q^s\equiv 1 \pmod 4\),  with \Cref{theorem: factorization of X^n-a for rad(n) mid q-1_4 nmid n or q = 1 mod 4} the factorization of \(X^n-a\) over \(\Fqto{s}\) is given by:
	\begin{equation*} \label{eq: factorization X^n-a over Fqto s}	
		X^{n}-a = \prod_{j=0}^{\frobenius  {d_1} s-1} \prod_{v \mid \frac {n_2}{\frobenius{d_2} s}} \prod_{\substack{i=0\\\gcd(i,v)=1}} ^{\frobenius{d_2} s-1}   (X^{\frac {n_1}{\frobenius {d_1} s} \cdot v}- \cyclgen{\frobenius {d_2} s}^i (\cyclgen{\frobenius{d_1} s}^{j} b)^{rv}),	
	\end{equation*} 
	where   \(b\in \Fqto{s}\) such that \(b^{\frobenius {d_1} s} = a\), \(r\cdot n_2 \equiv 1 \pmod{\ord(a)\cdot \frobenius {d_1} s}\) and \(\ord(\cyclgen{\frobenius {d_2} s}^i (\cyclgen{\frobenius{d_1} s}^{j} b)^{rv}) = \ord(a) \cdot \frobenius{d_1}{s} \cdot \frac {\frobenius{d_2}{s}}{\gcd(i, \frobenius{d_2}{s})}\) for all applicable \((j,v,i)\). We define 
	\begin{align*}
		\mathcal J = &\{(j,v,i): j \in \{0,\ldots, \frobenius {d_1}s-1\}, v \mid \frac {n_2}{\frobenius{d_2}{s}}, i  \in \{0, \ldots, \frobenius{d_2}{s}-1\}, ~ \gcd(i,v)=1 \}
	\end{align*} to be the set of all applicable \((j,v,i)\).  For all  \((j,v,i)\in \mathcal J\) we set	\(R_{j,v,i} := (X^{\frac {n_1}{\frobenius {d_1} s} \cdot v}- \cyclgen{\frobenius {d_2} s}^i (\cyclgen{\frobenius{d_1} s}^{j} b)^{rv})\) and \(c_i := \coeffdeg q{R_{j,v,i}}\).  With \Cref{theorem: factorization of X^n-a for rad(n) mid q-1_4 nmid n or q = 1 mod 4} holds   \[c_i = \ord_{\ord(\cyclgen{\frobenius {d_2} s}^i (\cyclgen{\frobenius{d_1} s}^{j} b)^{rv})}(q) = \ord_{\ord(a) \cdot \frobenius{d_1}{s} \cdot \frac {\frobenius{d_2}{s}}{\gcd(i, \frobenius{d_2}{s})}}(q),\]   which only depends on \(i\). Additionally, on \(\mathcal J\) we define an equivalence relation \(\approx\) as follows: \((j,v,i)\approx (\tilde j , \tilde v , \tilde i)\) if and only if \(R_{(\tilde j, \tilde v , \tilde i)} =  \frobenius{R_{(j,v,i)}}{m}\) for an integer \(0 \leq m  \leq c_i-1\). Meaning that \(R_{(\tilde j, \tilde v , \tilde i)}\) is a factor of the \(q\)-spin of  \(R_{(j,v,i)}\). Then from \Cref{theorem: corollary KK11 lemma 1} follows that the factorization of \(X^n-a\) over \(\Fq\) is given by:
	\(\prod_{(j,v,i)\in \mathcal J/\approx} \spin q {R_{(j,v,i)}}, \)
	where \(\spin q{R_{(j,v,i)}} = \prod_{m=0}^{c_i-1} \frobenius {R_{(j,v,i)}} m=: S_{(j,v,i)}\). Since \(R_{(j,v,i)}\) is a binomial, we can easily compute  \(S_{(j,v,i)}\):
	\[S_{(j,v,i,m)}=\sum_{l=0}^{c_{i}} X^{\frac {n_1}{\frobenius {d_1} s} \cdot v \cdot l}\cdot   (-1)^{c_i-l}\cdot  \sum_{\substack{\mathcal U \subseteq \{0, \ldots, c_{i}-1\}\\|\mathcal U|=c_i-l}} \prod_{u\in \mathcal U}  (\cyclgen{\frobenius {d_2} s}^{i\cdot q^m} (\cyclgen{\frobenius{d_1} s}^{j} b)^{rv})^{q^u},\]
	and from \Cref{theorem: ord(X^t-a)} follows that \(\ord(S_{(j,v,i)}) = \ord(R_{(j,v,i)}) = \frac {n_1}{\frobenius{d_1}s}\cdot v \cdot \ord(\cyclgen{\frobenius {d_2} s}^i (\cyclgen{\frobenius{d_1} s}^{j} b)^{rv}) = \frac {n_1}{\frobenius{d_1}s}\cdot v \cdot   \ord(a) \cdot \frobenius{d_1}{s} \cdot \frac {\frobenius{d_2}{s}}{\gcd(i, \frobenius{d_2}{s})} = n_1 \cdot v \cdot \ord(a) \cdot \frac {\frobenius{d_2}{s}}{\gcd(i, \frobenius{d_2}{s})}\).
	
	It remains to determine a representative system \(\mathcal J/\approx\).  For \(R_{(j,v,i)}\) and \(R_{(\tilde j, \tilde v , \tilde i)}\) to be of the same \(q\)-spin, the two polynomials  need to be of the same degree, which implies that  \(v = \tilde v\).  Furthermore, \(\cyclgen{\frobenius {d_2} s}^{\tilde i} (\cyclgen{\frobenius{d_1} s}^{\tilde j} b)^{rv} = \cyclgen{\frobenius {d_2} s}^{iq^m} (\cyclgen{\frobenius{d_1} s}^{j} b)^{rvq^m} \) for an integer  \(0\leq m\leq c_i-1\).  Since \(\gcd(\frobenius{d_2}{s}, \ord(a)\cdot \frobenius{d_1}{s})=1\), with \Cref{lemma: a^m=b^m iff a=b} \ref{item: Fact a b coprime order} this is equivalent to
	\begin{align}
		&\cyclgen{\frobenius{d_2}{s}}^{\tilde i} =\cyclgen{\frobenius{d_2}{s}}^{iq^m}\label{eq: condition on i}\\
		\text{and} \quad  & (\cyclgen{\frobenius{d_1} s}^{\tilde j} b)^{rv} = (\cyclgen{\frobenius{d_1} s}^{j} b)^{rvq^m} \label{eq: condition on j}
	\end{align} 
	for an integer \(0 \leq m \leq c_i-1\).	Condition (\ref{eq: condition on i})  is equivalent to \(\tilde i = i \cdot q^m \pmod {\frobenius {d_2} s}\). We  define \(\frobenius i m := i\cdot q^m \pmod{\frobenius {d_2} s}\) for all nonnegative integers \(m\). This is in fact an enumeration of the elements in \(\cyclocoset q {\frobenius{d_2}{s}} {i}\),  the \(q\)-cyclotomic coset  of \(i\) modulo \(\frobenius {d_2}s\). Since \(\cyclgen {\frobenius {d_2}s}^i\) is a proper element of \(\Fqto {t_i}\), \(\cyclocoset q {\frobenius{d_2}{s}} {i}\)  contains exactly \(t_i :=\ord_{\frac {\frobenius {d_2} s}{\gcd(i, \frobenius {d_2}s)}}(q)\) elements. Consequently, \(\frobenius i {t_i} = \frobenius i 0 = i\) and the elements \(\frobenius i 0, \frobenius i 1, \ldots, \frobenius i {t_i-1}\) are all distinct.
	
	Concerning equation (\ref{eq: condition on j}),  \(v\) divides \(n_2\) and \(rn_2\equiv 1 \pmod {\ord(a)\cdot \frobenius{d_1}{s}}\), which implies that \(\gcd(rv, \ord(a)\cdot \frobenius{d_1}{s})=1\). Then with \Cref{lemma: a^m=b^m iff a=b} \ref{item: Fact a^m b^m} equation (\ref{eq: condition on j})  is equivalent to \(\cyclgen{\frobenius{d_1} s}^{\tilde j} b = (\cyclgen{\frobenius{d_1} s}^{j} b)^{q^m}\). For any nonnegative integer \(m\) we define \[\frobenius j m := \tilde j \quad \Leftrightarrow \quad  \cyclgen{\frobenius {d_1}s} ^{\tilde j} b = (\cyclgen{\frobenius {d_1}s}^jb)^{q^m}.\] Note that the elements \(\{\cyclgen {\frobenius {d_1}s}^j b: 0 \leq j \leq \frobenius{d_1}{s}\}\) are the distinct roots of the polynomial \(X^{\frobenius{d_1}{s}}-a\in \FqX\) and our definition of \(\frobenius j m \) is well-defined. Since \(\rad(\frobenius{d_1}{s})\mid \ord(a)\), we can apply \Cref{lemma: deg of irred factors of X^n-a for rad(n) mid ord(a)} and the order of \(\cyclgen{\frobenius {d_1}s}^jb\) is  \(\ord(a)\cdot \frobenius {d_1} s\) for all \(0 \leq j \leq \frobenius {d_1}s\). We set \(s_1 := \ord_{\ord(a)\cdot \frobenius{d_1}s}(q)\)
	and every element \(\cyclgen {\frobenius {d_1}s}^jb\) is a proper element of \(\Fqto{s_1}\). Consequently, \(\frobenius j {s_1} = \frobenius j 0 = j\) and the integers \(\frobenius j 0 , \frobenius j 1, \ldots, \frobenius j {s_1-1}\) are all distinct. Additionally,  \Cref{lemma: deg of irred factors of X^n-a for rad(n) mid ord(a)} yields  that 
	\[s_1 = \begin{cases}
		\frac {\frobenius{d_1}s} {\frobenius {d_1}1} & \text{if \(4 \nmid \frobenius {d_1} s \) or \(q\equiv 1 \pmod 4\)},\\
		\frac {2{\frobenius{d_1}s}} {\frobenius {d_1}2} & \text{otherwise.}
	\end{cases}\]
	We define an equivalence relation \(\sim\) on the set of integers \(\{0, \ldots, \frobenius {d_1} s\}\), which follows naturally from our observations above: For \(j,\tilde j\in \{0,\ldots, \frobenius {d_1}s-1\}\) holds \(j \sim \tilde j\)  if and only if \(\tilde j = \frobenius j m\) for an integer \(0 \leq m \leq s_1-1\).
	
	Using our  definitions, \((j,v,i)\approx (\tilde j , \tilde v, \tilde i)\) holds if and only if \((\tilde j , \tilde v, \tilde i)=(\frobenius j m, v, \frobenius i m)\) for an integer \(0 \leq m \leq c_i-1\). Thus, the equivalence class of every  \((j,v,i)\in \mathcal J\) is:
	\begin{align*}
		[(j,v,i)] = \{(\frobenius j 0, v, \frobenius i 0), (\frobenius j 1, v, \frobenius i 1), \ldots, (\frobenius j {c_i-1}, v, \frobenius i {c_i-1})\}.
	\end{align*}
	It is our goal to find a unique representant \((\hat j,  v, \hat i)\) for every equivalence class \([(j,v,i)]\) in \(\mathcal J/\approx\).  The fact that \(s_1 = \ord_{\ord(a)\cdot \frobenius{d_1}{s}}(q)\) and  \(t_i = \ord_{\frac {\frobenius {d_2}s}{\gcd(i,\frobenius{d_2} s)}}(q)\) implies that \(c_i = \ord_{\ord(a)\cdot \frobenius{d_1}{s}\cdot\frac {\frobenius {d_2}s}{\gcd(i,\frobenius{d_2} s)} } = \lcm(s_2, t_i)\). Therefore the sequence \(\frobenius j 0, \ldots, \frobenius j {c_i-1}\) runs exactly \(\frac {c_i}{s_1}\) times through the distinct elements \(\frobenius j 0 , \ldots, \frobenius j {s_1-1}\). By definition, exactly one of these \(s_1\) elements is an element of \(\{0 , \ldots, \frobenius{d_1}s-1\}/\sim\). Thus, for the representant \((\hat j,  v, \hat i)\) of the equivalence class  \([(j,v,i)]\) we can assume that \(\hat j\in \{0,\ldots, \frobenius {d_1} s-1\}/\sim\). Without loss of generality, we assume that \(\hat j = \frobenius j 0\).
	
	Now we examine for which \(i,\tilde i\in \{0, \ldots, \frobenius {d_2}s-1\}\) the equivalence classes \([(\hat j, v,i)]\) and \([(j,v,\tilde i)]\) are equal. Obviously, \(\tilde i\) needs to be an element of \(\cyclocoset{q}{\frobenius {d_2}s}{i}\). If all \(\tilde i \in \cyclocoset{q}{\frobenius {d_2}s}{i}\) satisfy \([(\hat j, v,\tilde i)]=[(\hat j, v, i)]\), then we can select \(i \in \cyclorepsystem{q}{\frobenius{d_2}s}\) and uniquely describe every equivalence class. However, not all \(\tilde i\in \cyclocoset{q}{\frobenius {d_2}{s}} i \) satisfy \([(\hat j, v,\tilde i)]=[(\hat j, v, i)]\). In fact, since \(\frobenius j {s_1} = j\) and \(\frobenius{j}l \neq j\) for all \(1\leq l < s_1\), the equivalence classes of  \([(j,v,i)]\) and \([(j,v,\frobenius i m)]\) are equal if and only if \(m =s_1\cdot l \pmod {t_i}\) for a nonnegative integer  \(l\).  Since \(\{s_1\cdot l \pmod{t_i}: l\geq 0\}=\{\gcd(s_1,t_i) \cdot l \pmod{t_i}: l\geq 0\}\), this is equivalent to \(\tilde i = \frobenius im\) for an integer \(m\in \{\gcd(s_1,t_i)\cdot l \pmod{t_i}: l \geq 0\}\). From this fact follows directly, that  if \(\gcd(s_1,t_i)=1\) then \([(\hat j,v,i)]=[(\hat j,v,\frobenius im)]\) for all \(0 \leq m \leq t_i-1\). Otherwise
	\begin{align*}
		&[(\hat j,v,i)] = [(\hat j,v,\frobenius i {\gcd(s_1,t_i)})]= \ldots=[(\hat j,v,\frobenius i {\gcd(s_1,t_i)\cdot (\frac{c_i}{\gcd(s_1,t_i)}-1)})],\\
		&[(\hat j,v,\frobenius i1)] = [(\hat j ,v,\frobenius i {\gcd(s_1,t_i)+1})] = \ldots,\\
		&\vdots\\
		&[(\hat j,v,\frobenius i {\gcd(s_1,t_i)-1} )] = [(\hat j , v, \frobenius i {2 \gcd(s_1,t_i)-1})] = \ldots. 
	\end{align*}
	The equivalence classes \([((\hat j,v,i)],[(\hat j,v,\frobenius i1)] , \ldots, [(\hat j,v,\frobenius i {\gcd(s_1,t_i)-1} )]\) are all distinct. Thus, for every \(i \in \cyclorepsystem{q}{\frobenius {d_2}s}\) such that \(\gcd(v,i)=1\)  we put  \((\hat j, v,i), (\hat j, v,\frobenius{i}{1}), \ldots, \) \((\hat j,v,\frobenius i {\gcd(s_1,t_i)-1})\) in our representative system \(\mathcal J/\approx\), which we can describe as the following cartesian  product:
	\begin{align*}
		&\{0,\ldots, \frobenius{d_1}s-1\}/\sim \\
		&\times \{(v,\frobenius im): v\mid \frac{n_2}{\frobenius{d_2}s}, i \in \cyclorepsystem{q}{\frobenius{d_2}{s}}, \gcd(i,v)=1, 0 \leq m \leq  \gcd(s_1,t_i)-1\}.
	\end{align*}
	For any \(i \in \{0, \ldots, \frobenius {d_2}s-1\}\) and any positive integer \(t\) holds \(t_i = t\) if and only if \(t\) is the smallest positive integer such that \(\frac {\frobenius {d_2}s}{\gcd(i, \frobenius {d_2}s)} \) divides \(q^t-1\). This is true only if \(\frac {\frobenius {d_2}s}{\gcd(i, \frobenius {d_2}s)}\mid \gcd(\frobenius {d_2}s, q^t-1) = \frobenius {d_2}t\). Since \(\frobenius {d_2}t\) is a divisor of \(\frobenius {d_2}s\), this is equivalent to \(\frac {\frobenius {d_2}s}{\frobenius{d_2}t} \mid i\). Thus, \(t_i = \min \{t\in \N: \frac {\frobenius {d_2}s}{\frobenius{d_2}t} \mid i\}\). Furthermore, for all \(\tilde i \in \cyclocoset{q}{\frobenius{d_2}s}{i}\) holds \(\gcd(i, \frobenius{d_2}s)= \gcd(\tilde i, \frobenius{d_2}s)\). Consequently, \(\ord(S_{(j,v,\tilde i)}) = \ord(S_{(j,v,i)})\) and \Cref{theorem: factorization X^n-a for gcd(n_q)=1} holds. 
\end{proof}}

\subsection{The factorization of $X^n-1$ and of the cyclotomic polynomial $\Phi_n$ for positive integers $n$ such that  $\gcd(n,q)=1$}

\label{subsection: factorization X^n-1 and cyclotomic polynomial}

\changed{There exist explicit  formulas for the factorization of \(X^n-1\) over \(\Fq\) for the case that \(\rad(n)\mid q^{w}-1\) where \(w=1\) (see \cite[Corollaries 1 and 2]{Brochero-MartinezGiraldo-VergaradeOliveira2015}), \(w\) is prime (see \cite[Theorems 3.2, 3.4 and 3.6]{WuYueFan2018}) or \(w\) is the product of two primes (see \cite[Theorems 3.3, 3.5, 3.8, 4.2, 4.4 and 4.7]{WuYue2021}).  With  the choice \(a=1\)  we derive the following closed formula for the factorization of \(X^n-1\) for any positive integer \(n\) such that \(\gcd(n,q)=1\) from \Cref{theorem: factorization X^n-a for gcd(n_q)=1}. \Cref{theorem: factorization of X^n-1 for gcd(n q)=1} covers all factorizations given in \cite{Brochero-MartinezGiraldo-VergaradeOliveira2015, WuYueFan2018} and \cite{WuYue2021}.
	
\begin{tcolorbox}[breakable]
	\begin{ntheorem}\label{theorem: factorization of X^n-1 for gcd(n q)=1}
		Let \(n\in\N\) such that \(\gcd(n,	q)=1\). Let \(w:=\ord_{\rad(n)}(q)\) and set  \(s:=w\) if \(4\nmid n\) or \(q^w\equiv 1 \pmod 4\), else set \(s:=2w\). Furthermore, for all positive integers \(t\)   we define \(\frobenius{d} t:= \gcd(n, q^t-1)\) and for every \(i \in \cyclorepsystem{q}{\frobenius {d}s}\) we set \(c_i :=  \min\{t\in \N: \frac{\frobenius{d}s}{\frobenius{d}t}\mid i\}\). 
		
		Then  the factorization of \(X^n-1\) into monic irreducible factors over \(\Fq\) is 
		\begin{equation*}
			\prod_{v\mid \frac {n}{\frobenius{d}{s}}} \; \prod_{\substack{i \in \cyclorepsystem{q}{\frobenius{d}{s}}\\ \gcd(i,v)=1}}   \left[\sum_{l=0}^{c_{i}} X^{v\cdot l}\cdot   (-1)^{c_i-l} \sum_{\substack{\mathcal U \subseteq \{0, \ldots, c_{i}-1\}\\|\mathcal U|=c_i-l}} \prod_{u\in \mathcal U}  \cyclgen{\frobenius {d} s}^{i\cdot q^u}\right],
		\end{equation*} 
		where for all \(v\mid \frac n {\frobenius ds}\) and all \(i \in \cyclorepsystem{q}{\frobenius ds}\) such that \(\gcd(i,v)=1\), the monic irreducible factor belonging to \((v,i)\) has degree \(vc_i\) and order \(v \cdot \frac {\frobenius d {s}}{\gcd(i,\frobenius d s)}\).
	\end{ntheorem}
\end{tcolorbox}	
	
	\begin{proof}
		We apply \Cref{theorem: factorization X^n-a for gcd(n_q)=1} for \(a = 1\). Then \(\ord(a)=1\) implies that \(1 =n_1 = \frobenius {d_1} t\) for all  positive integers \(t\) and \(n=n_2\). Furthermore, \(\cyclgen{\frobenius {d_1}s} = 1\) and  the element \(b\in \Fq\) such that \(b^{\frobenius {d_1}s}=a\) can be set to \(1\). Its order is obviously \(1\), which implies that \(s_1 = 1\). The representative system \(\{0, \ldots, \frobenius{d_1}s-1\}/\sim = \{0\}/\sim\) is \(\{0\}\) and we can set \(j=0\) for all irreducible factors. Furthermore, for all \(i \in \cyclorepsystem{q}{\frobenius {d_2}s}\) holds  \(\gcd(t_i,s_1)=1\). \changed{ Since \(a=1\) the positive integer \(r\) satisfies \(r=1\).}  Set \(\frobenius d s := \frobenius {d_2}s\). Then the factorization of \(X^n-1\) over \(\Fq\) is given by
		\[\prod_{v \mid \frac {n}{\frobenius {d}s}} \prod_{\substack{i \in \cyclorepsystem{q}{\frobenius {d}s}\\ \gcd(i,v)=1}} S_{(0,v,i,0)} = \prod_{v \mid \frac {n}{\frobenius {d}s}} \prod_{\substack{i \in \cyclorepsystem{q}{\frobenius {d}s}\\ \gcd(i,v)=1}} = \sum_{l=0}^{c_i} X^{vl} \cdot  (-1)^{c_i-l} \sum_{\substack{\mathcal U \subseteq \{0, \ldots, c_{i}-1\}\\|\mathcal U|=c_i-l}}  \prod_{u\in \mathcal U} (\cyclgen{\frobenius d s}^i)^{q^u}.\]
		\(S_{(0,v,i,0)}\) has degree \(v\cdot c_i\) and order \(v \cdot \frac{\frobenius ds} {\gcd(i, \frobenius ds)}\).
	\end{proof}

	Since   \(X^n-1\) is the product \(\prod_{d\mid n} \Phi_d\), the factorization of the \(n\)-th cyclotomic polynomial \(\Phi_n\) is given by the product of all monic irreducible factors of order \(n\) in the factorization of \(X^n-1\). Thus, the following result is a corollary of \Cref{theorem: factorization of X^n-1 for gcd(n q)=1}. The result covers all explicit factorizations given in \cite{FY2007,WangWang2012,TW2013,WZFY2017,Alshareef2018}. Even though the result follows  from our main theorem, we  give a direct proof. 
	
	\begin{tcolorbox}[breakable]
		\begin{ntheorem}\label{theorem: factorization cyclotomic polynomial}
			Let \(n\in\N\) such that \(\gcd(n,	q)=1\). Let \(w:=\ord_{\rad(n)}(q)\) and set  \(s:=w\) if \(4\nmid n\) or \(q^w\equiv 1 \pmod 4\), else set \(s:=2w\). Furthermore, we define \(\frobenius{d} s:= \gcd(n, q^s-1)\). 
			
			Then  the factorization of the cyclotomic polynomial \(\Phi_n\) into \(\frac {\varphi(\frobenius ds)}{s} \) monic irreducible factors of degree \(\frac n {\frobenius ds}\cdot s\) over \(\Fq\) is 
			\begin{equation*}
				\prod_{\substack{i \in \cyclorepsystem{q}{\frobenius{d}{s}}\\ \gcd(i,\frobenius ds)=1}}   \left[\sum_{l=0}^{s}X^{\frac{n}{\frobenius d s}\cdot l}\cdot   (-1)^{s-l} \sum_{\substack{\mathcal U \subseteq \{0, \ldots, s-1\}\\|\mathcal U|=s-l}} \prod_{u\in \mathcal U}  \cyclgen{\frobenius {d} s}^{i\cdot q^u}\right].
			\end{equation*} 	
		\end{ntheorem}
	\end{tcolorbox}

\begin{proof}
	Let \(\mathcal  I:= \{0 \leq i \leq \frobenius ds -1: \gcd(i,\frobenius ds)=1\}\). Then every primitive \(n\)-th root of unity \(\cyclgen n\) is a root of a binomial of the form \(X^{\frac n {\frobenius ds}}-\cyclgen {\frobenius ds}^i\in \Fqto s[X]\) for a primitive \(\frobenius ds\)-th root of unity \(\cyclgen{\frobenius ds}^i\) with \(i \in \mathcal I\). Note that for two distinct elements \(i\) and \(\tilde i\) of \(\mathcal I\),  there do not exist any common roots of  \(X^{\frac n {\frobenius ds}}-\cyclgen{\frobenius ds}^i\) and  \(X^{\frac n {\frobenius ds}}- \cyclgen {\frobenius ds}^{\tilde i}\). Furthermore, there exist exactly \(\varphi(\frobenius ds)\) polynomials of this form. 
	
	Let \(i \in \mathcal I\). We show  that the binomial \(X^{\frac n {\frobenius ds}}-\cyclgen {\frobenius ds}^i\) is irreducible over \(\Fqto s\). Every prime factor of \(\frac n {\frobenius ds}\) is also a prime factor of \(\frobenius ds\) because  \(\rad(n)\mid q^s-1\) and \(\frobenius ds = \gcd( n, q^s-1)\). Since \( \ord(\cyclgen{\frobenius ds}^i)=\frobenius ds\), holds  \(\rad(\frac n {\frobenius ds})\mid \ord(\cyclgen {\frobenius ds}^i)\) and  \(\gcd(\frac n {\frobenius ds}, \frac {q^s-1}{\ord(\cyclgen{\frobenius ds}^i)}) = 1\). The positive integer \(s\) was selected  so that  either \(4 \nmid n\) or \(q^s\equiv 1\pmod 4\). With \Cref{Serret1866: Theorem Irreducibility of X^n-a} follows that the polynomial \(X^{\frac n {\frobenius ds}}-\cyclgen {\frobenius ds}^i\) is irreducible over \(\Fqto s\). Consequently,  \(\ord_n(q) = \frac n {\frobenius ds} \cdot s\). The roots of \(X^{\frac n {\frobenius ds}}-\cyclgen{\frobenius ds}^i\) are \(\frac n {\frobenius ds}\) distinct primitive \(n\)-th roots of unity.  Since there exist exactly \(\varphi(\frobenius ds)\) irreducible binomials of this form, we obtain that \(\varphi(n) = \frac n {\frobenius ds} \cdot \varphi(\frobenius ds)\). The factorization of \(\Phi_n\)  into monic irreducible factors over \(\Fqto s\) is :
	\[\prod_{\substack{i=0\\\gcd(i,\frobenius ds =1)}}^{\frobenius ds-1} (X^{\frac n {\frobenius ds}}-\cyclgen {\frobenius ds}^i).\]
	With \Cref{theorem: corollary KK11 lemma 1} the factorization of \(\Phi_n\) is given by \(\prod_{\substack{0 \leq i \leq \frobenius ds-1\\\gcd(i,\frobenius ds)=1}} \spin q {X^{\frac n {\frobenius ds}}-\cyclgen{\frobenius ds}^i}\). It remains to prove that for every \(i \in \mathcal I\), the primitive \(\frobenius ds\)-th root of unity \(\cyclgen {\frobenius ds}^i\) is a proper element of \(\Fqto s\). The smallest integer \(t\) such that \(\frobenius ds \mid q^t-1\) satisfies \(t = \ord_{\frobenius ds}(q)\geq \ord_{\rad(\frobenius ds)} (q)=  \ord_{\rad(n)}(q)  = w\) and from the fact that \(\frobenius ds\) divides \(q^s-1\) follows that \(t=\ord_{\frobenius ds} \leq s\). If \(s=w\), then \(t= \ord_{\frobenius ds}(q)=s\). If \(s=2w\), then \(4\mid n\) and \(q^w\equiv 3 \pmod 4\). Thus, with   \Cref{fact: gcd(n q^2-1) for q=3 mod 4} (\ref{item: Fact q^2-1 q equiv 3 mod 4 d_2}) holds \(\gcd(n,q^s-1) =\gcd(n,q^w-1) \cdot  2^{1+\min\{\nu_2(\frac n4), \nu_2(\frac {q+1}2)\}}\). Consequently, \(2 \mid \frac {\frobenius ds}{\frobenius dw}\) which implies that \(\frobenius ds \nmid q^w-1\) and \(s = \ord_{\frobenius ds}(q)\).  Thus, \(\coeffdeg q {X^{\frac n {\frobenius ds}}-\cyclgen {\frobenius ds }^i} = s\). Since \(X^{\frac n {\frobenius ds}}-\cyclgen {\frobenius ds }^i\) is a binomial, its \(q\)-spin can easily be computed:
	\[\spin q {X^{\frac n {\frobenius ds}} -\cyclgen{\frobenius ds}^i} =  \sum_{l=0}^{s}X^{\frac{n}{\frobenius d s}\cdot l}\cdot   (-1)^{c_i-l} \sum_{\substack{\mathcal U \subseteq \{0, \ldots, c_{i}-1\}\\|\mathcal U|=c_i-l}} \prod_{u\in \mathcal U}  \cyclgen{\frobenius {d} s}^{i\cdot q^u}.\]
\end{proof}}

\section{The factorization of $f(X^n)$}
\label{section: factorization of f(X^n)}

In this section we derive \Cref{theorem: factorization of f(X^n) for gcd(n q)=1}, a closed formula for the factorization of \(f(X^n)\) over \(\Fq\) for any monic  irreducible polynomial \(f\in \FqX\), \(f\neq X\), and any positive integer \(n\) such that \(\gcd(n,q)=1\), from our main theorem. The polynomial \(f=X\) is of no interest, since its composition with \(X^n\) always factors as \(f(X^n) = (X)^n\).	 Throughout this section let  \(f\in \FqX\) be a monic irreducible polynomial  of degree \(k\), \(f\neq X\), and \(\alpha\in \Fqk\) be a root of \(f\).  Then \Cref{theorem: factorization X^n-a for gcd(n_q)=1} yields the factorization of \(X^n-\alpha\) over \(\Fqk\). Therefore, the conditions from \Cref{section: factorization of X^n-a} that are posed on \(q\)   are conditions on \(q^k\) in this section. With \Cref{Mullin2010: Lemma 13} the factorization of \(f(X^n)\) is then given by the \(q\)-spins of all irreducible factors of \(X^n-\alpha\) over \(\Fqto k\).

\begin{nremark}\label{remark: gcd(n,q)>1 for f(X^n)}
	As in \Cref{section: factorization of X^n-a}, we can assume  \(\gcd(n,q)=1\) without loss of generality. Indeed, if \(\gcd(n,q)>1\), then \(n=\tilde n \cdot \char(\Fq)^l\) and there exists an element \(\gamma\in \Fqto k\) such that \(\gamma^{\char(\Fq)^l}=\alpha\). Let  \(\prod_R R\)  be the factorization of \(X^{\tilde n}-\gamma\) over \(\Fqto k\), then \((\prod_R R)^{\char(\Fq)^l}\) is the factorization of \(X^n-\alpha\) over \(\Fqto k\). With \Cref{Mullin2010: Lemma 13} the factorization of \(f(X^n)\) over \(\Fq\) then is given by \((\prod_R \spin q {R})^{\char(\Fq)^l}\). 
	
	Furthermore, we can asssume that \(f\) is monic without loss of generality. Indeed, if \(f=\sum_{i=0}^k a_i X^i\) is not monic, then \(a_k \in \Fq^\ast\setminus \{1\}\) and \(f = a_k \cdot g\) for a monic  polynomial \(g\in \FqX\). If \(\prod_{R\in \mathcal R} R\) is the factorization of \(g\in \FqX\) over \(\Fq\), then \(a_k \cdot \prod_{R\in \mathcal R}R\) is the factorization of \(f\) over \(\Fq\).
\end{nremark}

\setcounter{theorem}{19}
\begin{remark}
	The factorization of \(f(X^n)\) for the case \(\gcd(n,\ord(\alpha)\cdot k)=1\) and \(\rad(n)\mid q-1\) or \(\rad(n)\mid q^w-1\) for a prime \(w\) is given    in \cite{Brochero-MartinezReisSilva-Jesus2019}. In this remark we discuss the implications of this condition so that the differences between the existing and our results become clear. If \(\gcd(n,\ord(\alpha))=1\), then  there exists a positive integer \(r\) such that \(rn \equiv 1 \pmod {\ord(\alpha)}\) and  the element \(\alpha^r\) satisfies \((\alpha^{r})^n=\alpha\). Thus, this is a subcase of the case that there exists an element \(\beta\) in \(\Fqto{k}\) such that \(\beta^n = \alpha\).  The authors of \cite{Brochero-MartinezReisSilva-Jesus2019} claim that their results are generalizations of the results given in \cite{Brochero-MartinezGiraldo-VergaradeOliveira2015} and \cite{WuYueFan2018}.  Though this is technically true, we would like to point out that with \Cref{theorem: factorization of f(X^n) if beta^n=alpha} these factorizations could easily be obtained from the factorizations of \(X^n-1\) in \cite{Brochero-MartinezGiraldo-VergaradeOliveira2015, WuYueFan2018}. 	
	
	The condition \(\gcd(n,k)=1\)  implies that the roots of unity appearing in the factorization of \(f(X^n)\) (or in fact of \(X^n-\alpha\))  are elements of \(\Fq\) and not \(\Fqto k \). Since these roots of unity are determined by  \(\frobenius d k \), this can be seen from the following lemma and \Cref{theorem: d^(k)=d^(1)}, which is a corollary of \Cref{Bey77: Proposition 1 - nu_p(q^m-1)}.
\end{remark}

 \begin{lemma}[{\cite[Proposition 1]{Beyl1977}}]
 	\label{Bey77: Proposition 1 - nu_p(q^m-1)}
 	Let \(p\) be a prime such that \(p \mid (q-1)\) and let \(m\) be a positive integer. Then 
 	\begin{enumerate}[(i)]
 		\item If \(p\) is odd, then \(\nu_p(q^m-1) = \nu_p(q-1) + \nu_p(m).\)
 		\item If \(p=2\), then \(\nu_2(q^m-1)=\begin{cases}
 			\nu_2(q-1) & \text{ if } m \text{ is odd,}\\
 			\nu_2(q-1)+ \nu_2(m)+ \nu_2(q+1)-1& \text{ if } m \text{ is even,}
 		\end{cases}\)
 	\end{enumerate}
 	so that \(\nu_p(q^m-1) = \nu_p(q-1)+\nu_p(m)\) if \(q\equiv 1 \mod 4\).
 \end{lemma}
 
 \begin{corollary}\label{theorem: d^(k)=d^(1)}
 	Let \(n\in \N\) such that \(\rad(n)\mid q-1\) and \(\gcd(n,k)=1\). For a positive integer \(w\) we set  \(\frobenius{d}{w} := \gcd(n, q^w-1)\). Then
 	\begin{enumerate}[(i)]
 		\item \(\frobenius d {km} = \frobenius d m \) for every positive integer \(m\),
 		\item \(\frobenius d {2k} = 2^{l}\cdot \frobenius d 1\), where \(l = \min\{\max\{0,\nu_2(n)-\nu_2(q-1)\}, \nu_2(q+1)\}\).
 	\end{enumerate}
 \end{corollary}

 \begin{proof}
 	\begin{enumerate}[(i)]
 	\item Let \(p\) be a prime divisor of \(n\), then if \(p\) odd, we have \(\nu_p(q^{km}-1)=\nu_p(q^m-1)+\nu_p(k)= \nu_p(q^m-1)\). If \(p=2\), then \(k\) is odd, because \(\gcd(n,k)=1\), and \(\nu_2(q^{km}-1)=\nu_2(q^m-1)\).
 		\item Let \(p\) be a prime divisor of \(n\), then if \(p\) odd, we have \(\nu_p(q^{2k} -1)= \nu_p(q-1) + \nu_p(2k) = \nu_p(q-1)\), because \(p\neq 2\) and \(\gcd(n,k)=1\). If \(p=2\), then \(\nu_2(q^{2k}-1) = \nu_2(q-1) + \nu_2(2k) + \nu_2(q+1) -1 = \nu_2(q-1)+1 + \nu_2(q+1)-1=\nu_2(q-1)+\nu_2(q+1)\).
 	\end{enumerate}
 \end{proof}

\begin{remark}
	In \cite[Theorem 4.1]{Brochero-MartinezReisSilva-Jesus2019} the authors do not mention the condition \(\gcd(n, \ord(\alpha)\cdot k)=1\). However, without this condition the statement of the theorem would be false. For example with \Cref{Bey77: Proposition 1 - nu_p(q^m-1)}  the greatest common divisor \(d = \frobenius{d}{2}\) would equal \( \frobenius d 1 \cdot \gcd(n,k) \cdot 2^l\), where \(l := \min \{\nu_2(\frac n2), \nu_2(q+1)\}\) and not satisfy the equation \(d= 2^l \cdot \frobenius d 1\). Furthermore, the authors prove this result using results from their Section 3 and the condition \(\gcd(n,\ord(\alpha)\cdot k)\) is set throughout this section. 
\end{remark}

We use the following lemma in the proof of  \Cref{theorem: factorization of f(X^n) for gcd(n q)=1} to simplify our formula for the factorization of \(f(X^n)\). It shows that the \(q\)-spin is transitive. 

\setcounter{theorem}{25}

\begin{lemma}\label{theorem: spin(spin(g))}
	Let \(k,s\) be positive integers and \(g\) be an irreducible polynomial over \(\Fqto{ks}\) such that  \(k \mid \coeffdeg  qg\), then the \(q\)-spin of \(g\) satisfies
	\[\spin q g = \spin q {\spin{q^k}{g}}.\]
\end{lemma}

\begin{proof}
	Since \(k \mid \coeffdeg q g\), there exists a divisor  \(d\) of \(s\) such that \(\coeffdeg q g=kd\) and  \(\coeffdeg {q^k} g=d\). Then if \(g=\sum_{i=0}^{m} a_i X^i\),
	\begin{align*}
		\spin q {\spin{q^k}{g}} &= \spin q {\prod_{j=0}^{d-1} \sum_{i=0}^m a_i^{q^{kj}} X^i } = \prod_{l=0}^{k-1} \prod_{j=0}^{d-1} \sum_{i=0}^{m} a_i ^{q^{kj+l}} X^i \\
		&=  \prod_{j=0}^{kd-1}\sum_{i=0}^{m} a_i^{q^j} X^i = \spin q g ,
	\end{align*}
	because \(q^{l}\) is a power of \(\char(\Fq)\) and for the coefficients of the polynomial \(\prod_{j=0}^{d-1} \sum_{i=0}^{m} a_i ^{q^{kj+l}} X^i\) we can write i.e. \((a_i^{q^{kj}}+a_{\tilde i} ^{q^{k\tilde j}})^{q^l} = a_i^{q^{kj+l}}+a_{\tilde i}^{q^{k\tilde j + l}}\).
\end{proof}

The next theorem gives a closed formula for the factorization of \(f(X^n)\) into monic irreducible polynomials over \(\Fq\) for any positive integer \(n\) such that \(\gcd(n,q)=1\) and any monic irreducible polynomial \(f\in \FqX\) such that \(f\neq X\). It covers all factorizations given in \cite{Brochero-MartinezReisSilva-Jesus2019}.

\changed{
\begin{tcolorbox}
\begin{theorem}\label{theorem: factorization of f(X^n) for gcd(n q)=1}
	Let \(f\in \FqX\), \(f\neq X\), be a monic irreducible polynomial of degree \(k\) and \(\alpha\in \Fqk^\ast\) be a root of \(f\). Furthermore, let  \(n\) be a positive integer such that \(\gcd(n,	q)=1\) and  \(n=n_1 \cdot n_2,\) where \(\rad(n_1)\mid \ord(\alpha)\) and \(\gcd(n_2, \ord(\alpha))=1\). Let \(w=\ord_{\rad(n)}(q^k)\) and set  \(s:=w\) if \(4\nmid n\) or \(q^{kw}\equiv 1 \pmod 4\), else set \(s:=2w\). For all positive integers \(t\)   we set \(\frobenius{d_1} t := \gcd(n_1, \frac {q^{kt}-1}{\ord(a)})\) and \(\frobenius{d_2} t:= \gcd(n_2, q^{kt}-1)\). If \(4 \nmid \frobenius {d_1} s \) or \(q^k\equiv 1 \pmod 4\), then \(s_1:=\frac {\frobenius{d_1}s} {\frobenius {d_1}1}\), otherwise  \(s_1:=  \frac {2{\frobenius{d_1}s}} {\frobenius {d_1}2} \). For every \(i \in \cyclorepsystem{q}{\frobenius {d_2}s}\) we set \(t_i :=  \min\{t\in \N: \frac{\frobenius{d_2}s}{\frobenius{d_2}t}\mid i\}\)  and \(c_i:= \lcm(t_i,s_1)\). 
	
	Then there exists \(\beta\in \Fqto{ks}\) such that \(\beta^{\frobenius{d_1}{s}} = a\) and the factorization of \(f(X^n)\) into monic irreducible factors over \(\Fq\) is 
	\begin{equation*}
		\prod_{j\in \{0,\ldots, \frobenius{d_1}{s}-1\}/\sim } \; \prod_{v\mid \frac {n_2}{\frobenius{d_2}{s}}} \; \prod_{\substack{i \in \cyclorepsystem{q^k}{\frobenius{d_2}{s}}\\ \gcd(i,v)=1}} \prod_{m=0}^{\gcd(t_i,s_1)-1}  S_{(j,v,i,m)},\; 
	\end{equation*} 
	and for all applicable \((j,v,i,m)\) the monic irreducible polynomial  \(S_{(j,v,i,m)}\) of degree \(k \cdot \frac {n_1}{\frobenius{d_1}s} \cdot v \cdot c_i\) and order \(\ord(\alpha) \cdot n_1 \cdot v\cdot   \frac {\frobenius{d_2}{s}}{\gcd(i, \frobenius{d_2}{s})}\) is defined as  \begin{equation*}
		S_{(j,v,i,m)}=\sum_{l=0}^{kc_{i}} X^{\frac {n_1}{\frobenius {d_1} s} \cdot v \cdot l}\cdot   (-1)^{kc_i-l} \sum_{\substack{\mathcal U \subseteq \{0, \ldots, kc_{i}-1\}\\|\mathcal U|=kc_i-l}} \prod_{u\in \mathcal U}  (\cyclgen{\frobenius {d_2} s}^{i\cdot q^m} (\cyclgen{\frobenius{d_1} s}^{j} \beta)^{rv})^{q^u},
	\end{equation*} where \(r\) is a  positive integer satisfying \(r=1\) if \(f=X-1\) or  \(rn_2 \equiv 1 \pmod{\ord(\alpha)\cdot \frobenius{d_1}{s}}\) otherwise. 	Furthermore, for all \(j,\tilde j\in \{0 , \ldots, \frobenius {d_1} s-1\}\) holds \(j\sim \tilde j\) if and only  if \(\cyclgen{\frobenius{d_1} s}^{\tilde j} = \cyclgen{\frobenius{d_1}{s}}^{j\cdot q^{km}}\beta^{q^{km}-1}\) for an integer \(0\leq m\leq s_1-1\).
\end{theorem}
\end{tcolorbox}

\begin{proof}
	The factorization of \(X^n-\alpha\) over \(\Fqto k\) is given by \Cref{theorem: factorization X^n-a for gcd(n_q)=1} as 
	\begin{equation*}
		\prod_{j\in \{0,\ldots, \frobenius{d_1}{s}-1\}/\sim } \; \prod_{v\mid \frac {n_2}{\frobenius{d_2}{s}}} \; \prod_{\substack{i \in \cyclorepsystem{q^k}{\frobenius{d_2}{s}}\\ \gcd(i,v)=1}} \prod_{m=0}^{\gcd(t_i,s_1)-1}  S_{(j,v,i,m)}
	\end{equation*}
	where \(S_{(j,v,i,m)}\) is the \(q^k\)-spin of the binomial \(R_{(j,v,i,m)}:=X^{\frac {n_1}{\frobenius {d_1}s}\cdot v} - \cyclgen {\frobenius {d_2}s}^{i\cdot q^m} (\cyclgen{\frobenius {d_1}s}^j \beta)^{rv}\) with  \(\coeffdeg{q^k}{R_{(j,v,i,m)}}=c_i\). The positive integer \(r\) satisfies \(r=1\) if \(\alpha=1\), which is equivalent to \(f = X-1\), or  \(rn_2 \equiv 1 \pmod{\ord(\alpha)\cdot \frobenius{d_1}{s}}\) otherwise.  Then with \Cref{Mullin2010: Lemma 13} the factorization of \(f(X^n)=f^Q\) for \(Q=\frac {X^n}{1} \in \FqXrational\) is given by 
	\begin{align*}
		\prod_{j\in \{0,\ldots, \frobenius{d_1}{s}-1\}/\sim } \; \prod_{v\mid \frac {n_2}{\frobenius{d_2}{s}}} \; \prod_{\substack{i \in \cyclorepsystem{q^k}{\frobenius{d_2}{s}}\\ \gcd(i,v)=1}} \prod_{m=0}^{\gcd(t_i,s_1)-1}  \spin q {S_{(j,v,i,m)}},
	\end{align*}
	and \(\coeffdeg{q}{S_{(j,v,i,m)}}=k\). With \Cref{theorem: spin(spin(g))} holds \(\spin q {S_{(j,v,i,m)}} = \spin q {R_{(j,v,i,m)}}\) and \(\coeffdeg{q}{R_{(j,v,i,m)}}=kc_i\). Since \(R_{(j,v,i,m)}\) is a binomial, its \(q\)-spin can easily be computed:
	\begin{align*}
		\sum_{l=0}^{kc_{i}} X^{\frac {n_1}{\frobenius {d_1} s} \cdot v \cdot l}\cdot   (-1)^{kc_i-l} \sum_{\substack{\mathcal U \subseteq \{0, \ldots, kc_{i}-1\}\\|\mathcal U|=kc_i-l}} \prod_{u\in \mathcal U}  (\cyclgen{\frobenius {d_2} s}^{i\cdot q^m} (\cyclgen{\frobenius{d_1} s}^{j} \beta)^{rv})^{q^u}
	\end{align*}
	The degree of the \(q\)-spin of \(R_{(j,v,i,m)}\) is \(kc_i\cdot \frac {n_1}{\frobenius{d_1}s} \cdot v\) and its order is the same as the order of \(S_{(j,v,i,m)}\), which is \(\ord(\alpha)\cdot n_1 \cdot v \cdot \frac {\frobenius {d_2}s}{\gcd(i,\frobenius {d_2}s)}\).
\end{proof}}

	\printbibliography

\end{document}